\documentclass[12pt,twoside,reqno,psamsfonts]{amsart}

\usepackage[colorlinks = true, citecolor = black, linkcolor = black, urlcolor = black, pdfstartview=FitH]{hyperref}  

\usepackage{graphicx}
\usepackage{amsmath}
\usepackage{amssymb}
\usepackage{enumerate}
\usepackage{verbatim}
\usepackage{url}

\numberwithin{equation}{section}
\newcommand{\EE}{\mathcal{E}}

\newcommand{\R}{\mathbb{R}}
\newcommand{\C}{\mathbb{C}}

\newcommand{\N}{\mathbb{N}}

\newcommand{\cN}{\mathcal{N}}

\newcommand{\cL}{\mathcal{L}}

\newcommand{\cZ}{\mathcal{Z}}

\newcommand{\e}{\mathbf{e}}
\renewcommand{\d}{\mathbf{d}}
\newcommand{\cF}{\mathcal{F}}
\newcommand{\cE}{\mathcal{E}}

\newcommand{\cR}{\mathcal{R}}

\newcommand{\cO}{\mathcal{O}}
\newcommand{\cW}{\mathcal{W}}

\newcommand{\cP}{\mathcal{P}}

\newcommand{\eps}{\varepsilon}

\newcommand{\dist}{\operatorname{dist}}

\newcommand{\Hd}{\dim_\mathrm{H}}

\renewcommand{\emptyset}{\varnothing}
\renewcommand{\epsilon}{\varepsilon}
\renewcommand{\rho}{\varrho}
\renewcommand{\phi}{\varphi}

\newcommand{\1}{\mathrm{\mathbf{1}}}

\renewcommand{\a}{\mathbf{a}}
\renewcommand{\b}{\mathbf{b}}
\renewcommand{\c}{\mathbf{c}}

\newcommand{\A}{\mathbf{A}}

\newcommand{\p}{\mathbf{p}}

\renewcommand{\hat}{\widehat}
\newcommand{\NN}{\mathcal{N}}

\renewcommand{\iint}{\int\hspace{-0.1in}\int}

\DeclareMathOperator{\\diam}{\diam}

\DeclareMathOperator{\diam}{diam}
\DeclareMathOperator{\spt}{spt}

\theoremstyle{plain}
\newtheorem{thm}{Theorem}[section]
\newtheorem{theorem}{Theorem}[section]

\newtheorem{lemma}[theorem]{Lemma}
\newtheorem{prop}[theorem]{Proposition}

\theoremstyle{definition}

\newtheorem{definition}[theorem]{Definition}

\newtheorem{notations}[theorem]{Notations}
\newtheorem{parameters}[theorem]{Parameters}

\theoremstyle{remark}
\newtheorem{remark}[theorem]{Remark}

\usepackage[usenames,dvipsnames]{xcolor}

\usepackage{subfigure}

\graphicspath{ {pictures/} }

\begin{document}

\title[Spectral gaps and overlaps]{Spectral gaps and Fourier dimension\\ for self-conformal sets with overlaps}

\author{Simon Baker}
\address{Department of Mathematical Sciences, Loughborough University, Loughborough, LE11 3TU, UK}
\email{simonbaker412@gmail.com}
\author{Tuomas Sahlsten}
\address{Department of Mathematics and Systems Analysis, Aalto University, Espoo, Finland \& Department of Mathematics, University of Manchester, Manchester, United Kingdom}
\email{tuomas.sahlsten@aalto.fi}

\thanks{S.B. is supported by an EPSRC New Investigator Award (EP/W003880/1). T.S. is supported by the Academy of Finland via the project \emph{Quantum chaos of large and many body systems}, grant Nos. 347365, 353738.} 

\begin{abstract}
We prove a uniform spectral gap for complex transfer operators near the critical line associated to \textit{overlapping} $C^2$ iterated function systems on the real line satisfying a Uniform Non-Integrability (UNI) condition. Our work extends that of Naud (2005) on spectral gaps for nonlinear Cantor sets to allow overlaps. The proof builds a new method to reduce the problem of the lack of Markov structure to average contraction of products of \textit{random} Dolgopyat operators. This approach is inspired by a disintegration technique developed by Algom, the first author and Shmerkin in the study of normal numbers. As a consequence of the method of the second author and Stevens, our spectral gap result implies that the Fourier transform of any non-atomic self-conformal measure decays to zero at a polynomial rate for any $C^{2}$ iterated function system satisfying UNI. This latter result leads to Fractal Uncertainty Principles with arbitrary overlaps.\end{abstract}

\maketitle

\section{Introduction}

\subsection{Spectral gaps of transfer operators with overlaps} The spectral theory of transfer operators on Banach spaces associated to hyperbolic dynamical systems provides a fundamental tool in mixing properties of Anosov flows \cite{Dolgopyat1,AGY,LiPan}, counting problems associated to periodic orbits (see e.g. \cite{PP} and references), renewal theory (e.g. \cite{Lalley,Li}), scattering resonances in quantum chaos \cite{Naud,Stoyanov,MageeNaud1,MageeNaud2, CalderonMagee}, and Fourier dimension and non-concentration estimates needed for sum-product bounds \cite{Li,SS}. This theory has mostly focused on transfer operators with a Markov partition, or in the language of iterated function systems (IFSs), with some separation condition. However, for many natural arising problems there is no Markov partition. For example in the study of higher dimensional dynamical systems such as solenoids and horseshoes giving rise to overlapping nonlinear hyperbolic- or parabolic IFSs with overlaps \cite{Solomyak2,Si2,Si1,P}, hyperbolic endomorphisms that are far from being automorphisms in their basic set \cite{Mi}, overlapping iterated function systems of fractional linear transformations arising from random matrix products \cite{Pincus,Lyons2}, and \textit{Bernoulli convolutions} that are self-similar measures on the real line with overlaps \cite{Shmerkin,Varju,Varju2}.


It is very difficult to analyse overlapping IFSs. When considering parametrised families of IFSs a \textit{transversality technique} can be used to overcome the issues with overlaps. This technique has been successfully applied to many problems on Bernoulli convolutions, self-similar measures and more general non-linear IFSs (e.g. \cite{Solomyak2,Si1,Si2,Solomyak}). On the other hand, when studying a \textit{specific} IFS other tools need to be developed. Recently there has been an influx of new ideas using additive combinatorics and multiscale analysis, in particular Bourgain's discretised sum-product theory. These ideas have led to remarkable breakthroughs in the study of dimension theory and geometric properties of attractors associated to overlapping IFSs (e.g. \cite{Hochman, Shmerkin, HochmanBaranyRapaport, HochmanRapaport}). These advances in overlapping IFSs have almost exclusively focused on linear systems of self-similar or self-affine type, and the case of truly non-linear systems remains elusive except for parametrised families using the transversality method \cite{Solomyak2}. This is because in the case of systems with non-linearity many of the methods in the linear case (e.g. using convolution methods) do not transfer over, and one has to restrict the class of maps, e.g. as in the work of Hochman and Solomyak \cite{HochmanSolomyakInventiones}. The general $C^2$ overlapping IFS theory is still largely unexplored. 

In this article we will make fundamental steps towards understanding the dynamics and geometry of overlapping $C^2$ IFSs by proving a new spectral gap theorem for complex transfer operators associated to nonlinear $C^2$ IFSs with arbitrary overlaps. Adapting ideas from the non-overlapping case, we then apply this new spectral gap theorem to prove new Fourier dimension bounds and Fractal Uncertainty Principles \cite{Dyatlov,DyatlovZahl,BD1,BD2} for systems with arbitrary overlaps.

\begin{figure}[ht!]
\includegraphics[scale=0.25]{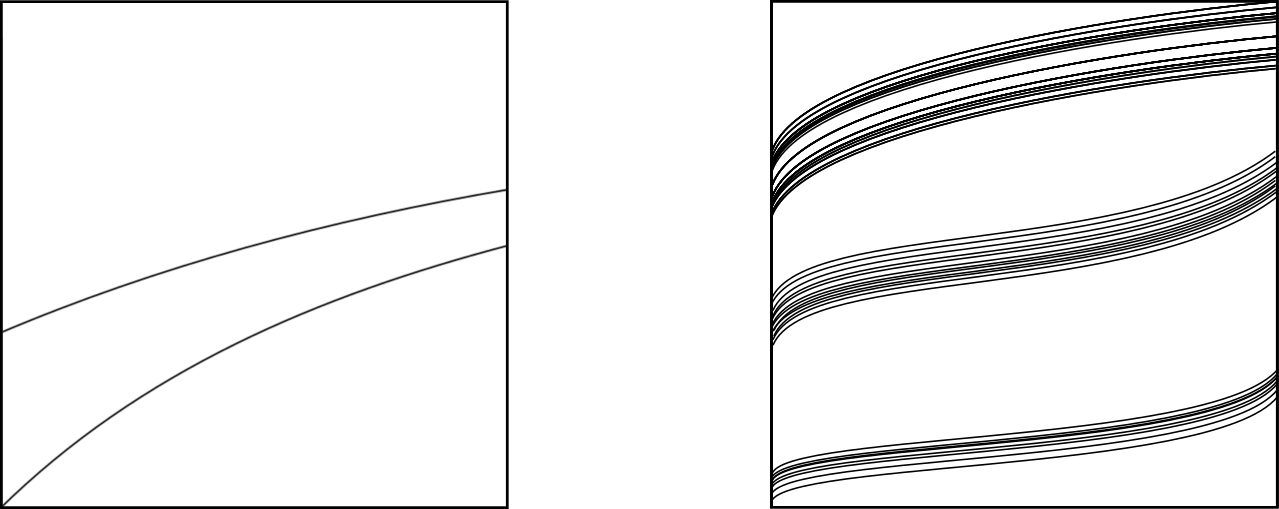}
\caption{Examples of overlapping non-linear $C^2$ IFSs on $I = [0,1]$, where left one is given by $\{\frac{x}{x+1},\frac{x+\alpha}{x+\alpha+1}\}$ with $\alpha = \frac{1}{2}$ considered by Lyons \cite{Lyons2}}
\label{fig:1}
\end{figure}

Our setting is an iterated function system $\Phi = \{\phi_a : a \in \A\}$ consisting of a finite number of $C^2$ contractions on an interval $I := [0,1]$, see e.g. Figure \ref{fig:1}. Then for a probability vector $(p_a)_{a \in \A}$ and $s = r+ib \in \C$, we associate a complex transfer operator $\cL_s : C^1(\R) \to C^1(\R)$ defined by
\begin{align}
\label{eq:operator}
\cL_{s} f(x) := \sum_{a \in \A} p_a |\varphi_{a}'(x)|^{s} f(\phi_a(x)), \quad x \in \R.
\end{align}
Such transfer operators arise naturally in the study of overlapping self-conformal measures of $\Phi$, which are eigenmeasures $\mu = \cL_0 \mu$. 

If the IFS $\Phi$ is sufficiently separated (e.g. it satisfies the Strong Separation Condition) and the branches satisfy the \textit{Uniform Non-Integrability} (UNI) condition introduced by Chernov \cite{Chernov} and Dolgopyat \cite{Dolgopyat1}, that is, 
there exists $c_1,c_2>0$ such that for all $n$ sufficiently large, there exists $\a,\b \in \A^{n}$ such that the compositions $\phi_\a = \phi_{a_1} \circ \dots \circ \phi_{a_n}$ and $\phi_\b = \phi_{b_1} \circ \dots \circ \phi_{b_n}$ satisfy
\begin{align}\label{eq:UNI}c_1\leq \left|  \frac{\phi_\a''(x)}{\phi_\a'(x)} -\frac{\phi_\b''(x)}{\phi_\b'(x)}\right|\leq c_2, \quad \text{for all } x \in K,\end{align}
then it goes back to the work of Naud \cite{Naud} and Stoyanov \cite{Stoyanov} who adapted \textit{Dolgopyat's method} \cite{Dolgopyat1}, that the transfer operators in \eqref{eq:operator} have a \textit{spectral gap} on $C^{1,b}(\R)$ for $|b|$ large enough and $|r|$ small enough. Here $K$ is the attractor of $\Phi$, i.e. the unique non-empty compact set satisfying $K=\cup_{a\in \A}\varphi_{a}(K)$, and $C^{1,b}(\R)$ is the Banach space of $C^1$ functions on $\R$ with the norm 
$$\|f\|_b := \|f\|_\infty + |b|^{-1}\|f'\|_\infty.$$ The spectral gap one obtains has useful applications to many problems, for instance scattering resonances \cite{Naud} and exponential mixing of Anosov flows \cite{Stoyanov} such as the Teichm\"uller flow \cite{AGY}.


The UNI condition is satisfied by many examples of roof functions. Informally it is saying that the IFS is uniformly far from being a linear IFS. It is implied by the \textit{total non-linearity} of the inverse branches $\phi_a$, that is, $\Phi$ not being $C^2$ conjugated to a linear system, see \cite[Claims 2.12, 2.13]{AHW2} and \cite[Claim 2.2]{AHW3} for a proof. In the special case where the inverse branches are analytic, then UNI is implied by $\Phi$ not being conjugated to a self-similar iterated function system. The UNI condition can also be replaced with weaker conditions, such as the \textit{non-local integrability property} (NLI) \cite{Naud} or the weaker \textit{local non-integrability condition} (LNIC) in higher dimensions \cite{Stoyanov}. 


If we introduce overlaps into the IFS $\Phi$, the methods of Dolgopyat \cite{Dolgopyat1}, Naud \cite{Naud} and Stoyanov \cite{Stoyanov} that rely on the Markov partition do not apply.  At the same time, most of the advances in the overlapping IFSs are focused on linear IFSs, so it is be unclear how the methods from linear IFSs would benefit the nonlinear case. When we have some true non-linearity in the system manifesting in the UNI condition, having spectral gaps for transfer operators could still be possible, but we need to overcome the overlapping structure with new ideas. 

In this work we indeed prove a uniform spectral gap for the operators \eqref{eq:operator} without any conditions on the overlaps:

\begin{theorem}
\label{thm:Linftybound}
Let $\Phi = \{\phi_a : a \in \A\}$ be a non-trivial uniformly contracting $C^2$ iterated function system satisfying the UNI condition \eqref{eq:UNI}. Then there exists $0 < \rho_0 < 1$ such that for $s = r+ib \in \C$ with $|r|$ sufficiently small and $|b|$ sufficiently large, the operator $\cL_{s}$ satisfies for all $n \in \N$ and $f \in C^1(\R)$:
$$ \|\cL_{s}^n f\|_b \lesssim \rho_0^{n} |b|^{1/2} \|f\|_b.$$
Thus there exists $0 < \delta < 1$ such that for all $|r|$ sufficiently small and $|b|$ sufficiently large, the spectral radius satisfies
$$\rho(\cL_s)\leq 1-\delta.$$
\end{theorem}

The novelty of this result comes from the way in which we overcome the arbitrary overlapping structure in the IFS. Our idea is to decompose the operator $\cL_s^n$ into a sum of \textit{random} compositions of transfer operators defined using a non-overlapping sub-IFS with nonlinearity. We then bound the norm of these compositions by using a family of \textit{random} Dolgopyat operators defined using the nonlinear sub-IFS. This decomposition idea is inspired by the work of Algom, the first author and Shmerkin \cite{ABS}, who disintegrated self-similar measures to study normal numbers. The way we build our decomposition suggests that our methods could be used in the study of IFSs that only contract on average \cite{Werner,Walkden}. Furthermore, as the theory of $C^2$ IFSs with overlaps is still in virgin territory, we believe the method we have built to prove Theorem \ref{thm:Linftybound} will provide an important step towards understanding the behaviour of $C^2$ IFSs with overlaps. 

\begin{remark}
We note that in the work \cite{AHW3} done simultaneously and independently of ours, Algom, Hertz and Wang also obtained a similar spectral gap theorem in the overlapping case, see \cite[Theorem 2.8]{AHW3}. They used this spectral gap theorem to prove an exponential decay rate in a renewal theorem leading to a Fourier decay theorem similar to the one we have in the next section (Theorem \ref{thm:main1}). We do not apply Theorem \ref{thm:Linftybound} to prove a renewal theorem, but instead use it to prove a non-concentration estimate that with a sum-product bound leads to a Fourier decay theorem. Interestingly, Algom, Hertz and Wang need the full strength of their Theorem 2.8 to prove their Fourier decay theorem. They need to consider $r<0$ to prove their renewal theorem, whereas to prove our non-concentration estimate we just need to consider the case where $r=0$.
\end{remark}

\begin{remark}
	Theorem \ref{thm:Linftybound} is formulated in terms of Bernoulli measures and their pushforwards. It seems likely that both proofs could be generalised to cover Markov measures and their pushforwards. The main obstacle that would need to be overcome is understanding how the Markov structure would interact with the partition of our IFS (see Proposition \ref{prop:IFS partition} for details of this partition). In particular, given an element $w$ in our partition, because we are now working with a Markov measure, the elements of the partition that may follow $w$ will now depend upon $w$. This in turn means that when we decompose our transfer operator as a sum of random transfer operators (see Lemma \ref{lemma:operator disintegration}), we would have to be more careful about which random transfer operators are allowed. If there is an element of our partition that is allowed to follow itself and satisfies an appropriate version of the UNI condition, then this element will play the role of $w^*$ in the statement of Proposition \ref{prop:IFS partition}, and our arguments should still work with only minor changes.
	
	Furthermore, it is natural to wonder whether the proof of Theorem \ref{thm:Linftybound} could be adapted to cover Gibbs measures and their pushforwards as in \cite{Stoyanov}. The authors expect that this is possible. However, our method for decomposing the transfer operator does not work for these more general measures, and so further ideas are needed. 
\end{remark}

\subsection{Fourier decay in overlapping $C^2$ IFSs}

Next, we want to move to an application of Theorem \ref{thm:Linftybound}, and in particular to the Fourier transforms of fractal measures. The study of Fourier transforms of fractal measures and their high-frequency asymptotics was historically initiated by questions on uniqueness of trigonometric series, metric number theory, Fourier multipliers and maximal operators defined by fractal measures (see e.g. \cite{JialunSahlsten1,JialunSahlsten2} for a historical overview). There were works on weaker average decay for Fourier transforms of self-similar measures by R. Strichartz \cite{Str} and M. Tsujii \cite{Tsujii}, and various works on specific constructions such as Fourier transforms of Bernoulli convolutions originating in Erd\"os' work \cite{Erdos}, random measure constructions e.g. using the Brownian motion \cite{FOS,FS}, and constructions in Diophantine approximation \cite{Kaufman}. These works suggested some form of pseudo-randomness of the underlying dynamical system should lead to the decay of the Fourier transform. This principle was verified in the article \cite{JordanSahlsten} by Jordan and the second author in the case of equilibrium states for the Gauss map where nonlinearity manifested in the distribution of the continuants of the continued fraction expansions, and then also in the subsequent article by Bourgain and Dyatlov \cite{BD1} on limit sets of Fuchsian groups. This latter article was motivated by proving a Fractal Uncertainty Principle to gain new information on scattering resonances in quantum chaos. 

Various thermodynamic, renewal theoretic and additive combinatoric techniques have been built which enable a systematic study of Fourier transforms of fractal measures using e.g. the under nonlinearity of the system. Since \cite{JordanSahlsten,BD1}, there has been a surge of activity in this topic in dynamics, metric number theory and fractal geometry to characterise measures with Fourier decay, such as for self-similar- and self-affine iterated function systems \cite{JialunSahlsten1,JialunSahlsten2,Solomyak,Bremont,VarjuYu,Rapaport}, self-conformal systems \cite{SS,AHW}, hyperbolic dynamical systems \cite{Leclerc1,Leclerc2,Leclerc3,Wormell}, fractal measures arising from random processes such as Brownian motion and Liouville quantum gravity \cite{FOS,FS,FalconerJin,ShmerkinSuomala}. 

The method used in \cite{SS} by the second author and Stevens to prove polynomial Fourier decay for certain self-conformal measures in dimension $1,$ was based upon the thermodynamic formalism method introduced in \cite{JordanSahlsten} combined with a corollary of a sum-product theorem as in \cite{BD1}. The required non-concentration assumption for the sum-product bound was verified using a spectral gap theorem for complex transfer operators due to Stoyanov \cite{Stoyanov}. Independently, Algom-Hertz-Wang \cite{AHW} proved that the Rajchman property holds for self-conformal measures with weaker assumptions but without polynomial Fourier decay. Similar ideas have been generalised by Leclerc \cite{Leclerc3} to hyperbolic attractors. These works leave open the possibility of overlaps and how the result would work in higher dimensions. There is motivation to study such problems, especially due to the need to generalise Fractal Uncertainty Principles \cite{BD1,BD2} to more general fractals arising from Anosov flows in order to optimise essential spectral gap bounds in \textit{variable} negatively curved manifolds to study quantum scattering problems in this situation, see Section \ref{sec:FUP} for more discussion.

In the higher dimensional case there has to be restrictions on the non-concentration of the measures on hyperplanes, see e.g. self-affine systems \cite{JialunSahlsten2} and the recent work by Khalil \cite{Khalil} on exponential mixing of the geodesic flow on a geometrically finite locally symmetric space of negative curvature with respect to the Bowen-Margulis-Sullivan measure, where Fourier decay results are studied under non-concentration on hyperplanes. In the overlapping \textit{self-similar} case, it is possible to obtain logarithmic Fourier decay \cite{JialunSahlsten1,Bremont,VarjuYu}, but the renewal theoretic method uses the Cauchy-Schwartz inequality in a way so that the non-concentration from purely derivatives is not strong enough to establish polynomial Fourier decay. 

Thanks to the spectral gap Theorem \ref{thm:Linftybound}, we will handle what happens when there are overlaps for self-conformal measures in the $C^2$ category:

\begin{theorem}\label{thm:main1}
Let $\mu$ be a non-atomic self-conformal measure associated to a $C^{2}$ iterated function system $\Phi$ on $\R$ satisfying the UNI condition \eqref{eq:UNI}. Then there exists $\alpha > 0$ such that
$$|\widehat{\mu}(\xi)| \lesssim |\xi|^{-\alpha}$$
for all $\xi \in \R$ with $|\xi| > 1$, where $\widehat{\mu}(\xi) := \int e^{-2\pi i \xi x} \, d\mu(x)$.
\end{theorem}

The way we approach bounding the Fourier transform of the measure $\mu$ in Theorem \ref{thm:main1} is based on iterating the self-conformal property of $\mu$ so that we can write an upper bound for $|\hat{\mu}(\xi)|^2$ that consists of an exponential sum over the regular blocks of words plus an error term that is small due to the large deviation bounds. This part is fundamentally the same as in \cite{SS}, which in turn was based upon combining the large deviation approach of Jordan and the second author \cite{JordanSahlsten} with the multiscale block decomposition by Bourgain and Dyatlov \cite{BD1}. To then control this exponential sum, we use, as in \cite{BD1,SS} a sum-product bound due to Bourgain \cite{Bourgain2010} that requires us to check a non-concentration hypothesis for the derivatives of the IFS. The main novelty comes in the proof of this  non-concentration property, which needs the new spectral gap Theorem \ref{thm:Linftybound} for the transfer operators $\cL_s$ with $\mathrm{Re}(s) = 0$ and $\mathrm{Im}(s) = c\xi$ for a suitable constant $c \in \R$. Thus we do not need information on $\cL_s$ outside of the critical line, and in fact here we only need an $L^\infty$ norm bound for $\cL_s^n f$ instead of a bound with respect to the $b$-norm. We note however that the proof of the $L^\infty$ norm bound at the critical line is the most non-trivial part of the proof of Theorem \ref{thm:Linftybound}.


\subsection{Fractal Uncertainty Principles and overlaps}\label{sec:FUP}

Finally, motivated by the work of Bourgain and Dyatlov \cite{BD1,BD2}, we discuss an application of Theorem \ref{thm:main1} to \textit{Fractal Uncertainty Principles} in quantum chaos, in particular, providing new examples where Markov structure can be avoided. Fractal Uncertainty Principles (FUPs) are a recently developed tool in harmonic analysis, which states that no function can be localised in both position and frequency near a fractal set, or more precisely: we say sets $X,Y \subset \R^d$ satisfy a \textit{Fractal Uncertainty Principle} at the scale $h > 0$ with exponent $\beta > 0$ and constant $C > 0$ if for all $f \in L^2(\R^d)$ with
$$ \{ \xi \in \R^d : \widehat{f}(\xi) \neq 0\} \subset h^{-1} Y \qquad \Rightarrow \qquad \|f\|_{L^2(X)} \leq Ch^{\beta} \|f\|_{L^2(\R^d)},$$
where $\widehat{f}(\xi) := \int_{\R^d} e^{-2\pi i x \cdot \xi } f(x) \, dx$, $\xi \in \R^d.$ When applied to $h$-neighbourhoods $X$ and $Y$ of fractals arising from hyperbolic dynamics, FUP has led to powerful applications in quantum chaos such as in bounding the essential spectral gaps and the $L^2$ mass of eigenfunctions of the Laplacian in open sets, and new control and observability theorems of PDEs \cite{DJ,DJN}. By a result of Bourgain and Dyatlov \cite{BD2}, \textit{porosity} (or Ahlfors-David regularity) of the sets $X$ and $Y$ in an interval of scales $[h,1]$ is enough to establish \textit{some} exponent $\beta > 0$ in the FUP, but quantifications especially for sets of dimension less than $1/2$ where additive combinatorics methods are used (e.g. by Dyatlov and Zahl \cite{DyatlovZahl} and Cladek and Tao \cite{CT}), require more structure from the fractal such as nonlinearity or curvature assumptions \cite{BD1}. 

If we consider systems \textit{without} porosity such as non-injective hyperbolic skew products with overlapping fibres \cite{Mi} or parabolic systems \cite{Lyons2} where holes may not appear uniformly at all scales, it would be interesting to see if FUP could be applied in such more general contexts. This could potentially have utility in quantum chaos related to such systems. In the following, we will consider FUP for sets $X$ and $Y$ that arise as neighbourhoods of fractals in $\R^d$ potentially without any porosity, but instead satisfy a Fourier decay condition and a mild Frostman regularity condition that is still possible even with arbitrary overlaps. 

We say a measure $\mu$ on $\R^d$ is $(C^-,\delta^-,C^+,\delta^+,h)$-\textit{Frostman} if:
\begin{itemize}
\item[(1)] For $r \in [h,1]$ and $x \in \R^d$ we have $\mu(B(x,r)) \leq C^+r^{\delta^+}$;
\item[(2)] For $r \in [h,1]$ and $x \in \spt \mu$ we have $\mu(B(x,r)) \geq C^-r^{\delta^-}$.
\end{itemize}
Here $\spt \mu$ is the support of the measure $\mu$. Note that all non-atomic self-conformal measures even with overlaps are $(C^-,\delta^-,C^+,\delta^+,h)$-Frostman for all small enough $h$, see e.g. \cite[Proposition 2.2]{FL} which has same proof in the self-conformal case using bounded distortions. The way Fourier decay connects to FUP can be observed in the following statement, it has a similar proof to that given in \cite{BD1} in the special case of limit sets of Fuchsian groups, but we extend it to ensure only the weaker Frostman condition is applied.
\begin{prop}\label{thm:fup}
For $j =1,2$, suppose $K_j = \spt \mu_j \subset \R^d$ are supports of $(C_j^-,\delta_j^-,C_j^+,\delta_j^+,h)$-Frostman measures. Assume also that for some $0 < \alpha \leq \delta_2^+ / 2$ we have:
$$|\widehat{\mu_2}(\xi)| \lesssim |\xi|^{-\alpha}, \quad |\xi| \leq \\diam(K_2) h^{-1}.$$
Let $X = K_1 + B(0,h)$ and $Y = K_2 + B(0,h)$. Then any $f \in L^2(\R^d)$ with
$$ \{ \xi \in \R^d : \widehat{f}(\xi) \neq 0\}  \subset h^{-1}Y \qquad \Rightarrow \qquad \|f\|_{L^2(X)} \lesssim_{C_1^-,C_2^\pm,\delta_1^-,\delta_2^-} h^{\frac{d}{2} - \frac{\delta_1^-}{2} - \frac{\delta_2^-}{2} + \frac{\alpha}{4}} \|f\|_{L^2(\R^d)}.$$
\end{prop}

We can now combine this with Theorem \ref{thm:main1} so that we obtain a wide class of non-porous and overlapping fractals such as basic sets of overlapping self-conformal sets satisfying FUP:

\begin{theorem}\label{cor:main}
Let $K_1,K_2 \subset \R$ be any non-trivial self-conformal sets for $C^2$ IFSs. Assume that the IFS associated to $K_2$ satisfies the UNI condition \eqref{eq:UNI}. Then there exists $\alpha > 0$ depending only on the IFS associated to $K_2$ such that FUP holds at the scale $h > 0$ for $X = K_1 + B(0,h)$ and $Y = K_2 + B(0,h)$ with $\beta = \frac{1}{2} - \frac{\delta_1^-}{2} - \frac{\delta_2^-}{2} + \frac{\alpha}{4}$, where 
$$\delta_j^- = \max\{\overline{\Hd}\, \mu : \mu \text{ is a self-conformal measure on } K_j\},$$
for $\overline{\Hd}\, \mu = \mathrm{ess\,sup}_{x \in \R} \limsup_{r \to 0} \log \mu(B(x,r))/\log r$.
\end{theorem}

Most of the proof of the Fourier decay theorem (Theorem \ref{thm:main1}) that implies Fractal Uncertainty Principle applies also in higher dimensions. In higher dimensions we would need a replacement for the sum-product bound, and a projective one by Li \cite{JialunRd} would be natural here. However, to prove the projective non-concentration for the derivatives would require us to have an assumption on avoiding concentration to hyperplanes. For example, in higher dimensions, FUP cannot work for even all porous sets, e.g. using the line segments $X_h = \R \times [-h,h]$ and $Y_h = [-h,h] \times \R$. In two dimensions, it is possible to obtain a Fractal Uncertainty Principle by using the Fourier decay of the Patterson-Sullivan measure like in the case of Fuchsian groups \cite{BD1} and in two dimensions by Li-Naud-Pan \cite{JialunNaudPan}, see also Leclerc's recent work \cite{Leclerc1} involving twisted transfer operators and bunched attractors \cite{Leclerc3}. Moreover, adapting Dolgopyat's method, which also lies at the heart of proving a spectral gap for complex transfer operators \cite{Naud,Stoyanov}, Backus-Leng-Tao \cite{BLT} proved a Fractal Uncertainty Principle for limit sets of Kleinian groups in $\mathbb{H}^d$ with exponent $d/2-\Hd F + \eps$, where $\Hd F$ is the Hausdorff dimension of the limit set. This generalised the 1D approach of Dyatlov and Jin \cite{DJ2} who obtained a similar result using Dolgopyat's method. A generalisation of Corollary \ref{cor:main} to higher dimensions applied to subshifts of finite type would allow for a higher dimensional Fractal Uncertainty Principle with a similar exponent as Backus, Leng and Tao. We also make note of a recent remarkable work of Cohen \cite{Cohen2} on proving a higher dimensional Fractal Uncertainty Principle for line porous fractals without any dimension assumptions. 

We believe a UNI assumption for all directions similar to the one considered in \cite{AGY} for the exponential mixing of the Teichm\"uller flow would provide a suitable analogue where the results of this paper would hold, see also \cite{Khalil}. We plan to investigate this in a future work.

\subsection*{Organisation of the article} In Section \ref{sec:symb} we go through the basic symbolic notation we need. In Section \ref{sec:naud} we prove the new spectral gap theorem. Then in Section \ref{sec:proofmain} we prove Theorem \ref{thm:main1} on Fourier decay using the spectral gap of transfer operators. Finally in Section \ref{sec:fup} we give the Fractal Uncertainty Principle argument in $\R^d$.

\subsection*{Notation} We collect here some notational conventions that we will adopt throughout this article. Given two real valued functions $f,g$ defined on a set $S$. We write $f\lesssim g$ if there exists a constant $c>0$ such that $f(x)\leq cg(x)$ for all $x\in S$. We write $f\sim g$ if $f\lesssim g$ and $g\lesssim f$. We will also on occasion write $f=\cO(g)$ to mean the same thing as $f\lesssim g$.

\section{Symbolic notations of $C^2$ IFSs}\label{sec:symb}

Let $\{\varphi_a : I \to I\}_{a\in \A}$, $I = [0,1]$, $\A$ finite, be a $C^2$ iterated function system (IFS) acting on $\mathbb{R}$ satisfying the following properties:
\begin{itemize}
\item[(1)] \textit{Uniform contraction}: There exists $1 < \gamma < \gamma_1$ such that for all $x \in I$ and $n \in \N$, if $(a_1,\ldots,a_n)\in \A^n$ then
$$\gamma_1^{-n} \lesssim |(\varphi_{a_1}\circ \cdots \circ \varphi_{a_n})'(x)| \lesssim  \gamma^{-n}.$$
\item[(2)] \textit{Bounded distortions}: For all $x,y \in I$ we have
 	$$\frac{|\varphi_a'(x)|}{|\varphi_a'(y)|} \leq B.$$
 	\item[(3)] \textit{Non-trivial}: The unique non-empty compact set $K$ satisfying $$K=\bigcup_{a\in \A}\varphi_{a}(K)$$ is not a singleton.
\item[(4)] \textit{Uniform Non-Integrability} (UNI): We say that $\Phi$ satisfies the uniform non-integrability condition if there exists $c_1,c_2>0$ such that for all $n$ sufficiently large, there exists $\a,\b \in \A^{n}$ such that the compositions $\phi_\a = \phi_{a_1} \circ \dots \circ \phi_{a_n}$ and $\phi_\b = \phi_{b_1} \circ \dots \circ \phi_{b_n}$ satisfy
\begin{align*} c_1\leq \left|  \frac{\phi_\a''(x)}{\phi_\a'(x)} -\frac{\phi_\b''(x)}{\phi_\b'(x)}\right|\leq c_2, \quad \text{for all } x \in K.\end{align*}
\end{itemize}

Given a probability vector $\p=(p_a)_{a\in \A}$ ($0 < p_a < 1$ and $\sum_{a \in \A} p_a = 1$), there exists a unique Borel probability measure $\mu_{\p}$ satisfying
$$\mu_\p = \sum_{a \in \A} p_a f_a \mu_p.$$
$\mu_{\p}$ is called a \textit{self-conformal measure}. When the choice of $\p$ is implicit we will simply denote $\mu_\p$ by $\mu$. 

We now take the opportunity to introduce some tree notation. We let $\A^*=\cup_{n=1}^{\infty}\A^n$ denote the set of finite words over the alphabet $\A$. Given $\a=(a_1,\ldots, a_n)\in \A^*$ we let $$\varphi_{\a}=\varphi_{a_1}\circ \cdots \circ \varphi_{a_n}\qquad\textrm{ and }\qquad p_{\a}=\prod_{i=1}^{n}p_{a_i}.$$ We also let $$[\a]:=\{\b\in \A^{\N}:b_i=a_i \textrm{ for }1\leq i\leq n\}$$ denote the cylinder set associated to $\a.$ We let $\pi:\A^{\N}\to K$ be the usual projection map given by $$\pi(\a)=\lim_{n\to\infty}(\varphi_{a_1}\circ \cdots \circ \varphi_{a_n})(0).$$ Given a probability vector $\p=(p_a)_{a\in \A}$ we denote by $m_{\p}:=\p^{\mathbb{N}}$ the product measure on $\A^{\mathbb{N}}$. $m_{\p}$ are $\mu_{\p}$ are connected via the equation $\mu_{\p}=\pi \mu_{\p}$. 

\section{Proof of the spectral gap theorem}\label{sec:naud}

We cannot directly apply the argument of Naud \cite{Naud} due to the potential overlaps coming from the IFS. To overcome this issue we use a disintegration argument due to Algom, the first author and Shmerkin \cite{ABS}. In this paper the authors showed that one could disintegrate a self-similar measure $\mu$ into measures that looked like self-similar measures for well separated IFSs. We employ a similar idea, however instead of disintegrating the measure $\mu$, we in effect ``disintegrate" the transfer operator and introduce a class of \textit{random} Dolgopyat operators.

\subsection{Partitioning the IFS and random transfer operators.}

The following proposition guarantees the existence of a useful partition of our IFS. Roughly speaking, this partition splits our IFS into non-trivial sub-IFSs each one of which is well separated. Moreover, there exists one special sub-IFS that satisfies a suitable uniform non-integrability condition. 

\begin{prop}
	\label{prop:IFS partition}
Let $\{\varphi_{a}\}_{a\in \A}$ be a non-trivial IFS satisfying the UNI condition. Then there exists $N\in\mathbb{N}$ and $w^*,w_1,\ldots, w_m\subset \A^N$ such that the following properties are satisfied:
\begin{enumerate}
	\item[(1)] $w^*\cup w_1\cup \cdots \cup w_m=\A^{N}. $ Moreover, this union is disjoint.
	\item[(2)]  $ \sharp w_i\in \{2,3\}$ for $1\leq i\leq m$.
	\item[(3)]  For any $1\leq i\leq m$, for distinct $\a,\b\in w_i$ we have $$\varphi_{\a}(I)\cap \varphi_{\b}(I)=\emptyset.$$
	\item[(4)]  $w^*=\{\alpha_1,\alpha_2\}$ and these words satisfy:
	\begin{itemize}
		\item[a.] $$\varphi_{\alpha_1}(I)\cap \varphi_{\alpha_2}(I)=\emptyset.$$
		\item[b.] There exists $c_1,c_2,\delta>0$ such that for all $x\in \{x:d(x,K)<\delta\}$ and $l\in\mathbb{N}$ we have $$c_1\leq \left|\frac{\varphi_{\alpha_{1}^{l}}''(x)}{\varphi_{\alpha_{1}^{l}}'(x)}-\frac{\varphi_{\alpha_{2}^{l}}''(x)}{\varphi_{\alpha_{2}^{l}}'(x)}\right|\leq c_2.$$
		
		\item[c.] $p_{\alpha_1}=p_{\alpha_2}.$
	\end{itemize} 
	
\end{enumerate}
\end{prop}

\begin{proof}
We begin our proof by remarking that for any $\a\in \A^*$ we have 
\begin{equation}
	\label{e:differentUNI}
	\frac{d}{dx}\log|\varphi_{\a}'(x)|=\frac{\varphi_{\a}''(x)}{\varphi_{\a}'(x)}.
\end{equation}
Therefore the UNI condition \eqref{eq:UNI} is equivalent to the following: there exists $c_1,c_2>0$ such that for all $M$ sufficiently large, there exists $\a,\b\in \A^M$ such that 	
\begin{equation}
	\label{eq:EquiUNI}
	c_1\leq \left| \frac{d}{dx}(\log|\varphi_{\a}'(x)|-\log|\varphi_{\b}'(x)|)\right|\leq c_2, \qquad \text{for all }x\in K.
	\end{equation}
 What makes \eqref{eq:EquiUNI} easier to work with is the following useful identity that follows from two applications of the chain rule: for any $\a,\b,\c,\d\in \A^*$ and $x\in I$ we have
\begin{align*}	\frac{d}{dx}(\log|\varphi_{\a\c}'(x)|-\log|\varphi_{\b\d}'(x)|)	=&\frac{d}{dx}(\log|\varphi_{\c}'(x)|-\log|\varphi_{\d}'(x)|)\\
	+&\varphi_{\c}'(x)\left(\frac{d}{dx}\log|\varphi_{\a}'|\right)(\varphi_{\c}(x))-\varphi_{\d}'(x)\left(\frac{d}{dx}\log|\varphi_{\b}'|\right)(\varphi_{\d}(x)).
\end{align*} Using this identity and appealing to a bounded distortion argument, it can be shown that there exists $C>0,$ such that for any $\a,\b,\c,\d\in \A^*$ and $x\in I$ we have
\begin{equation}
	\label{e:transferring nonlinearity}
	\left|\frac{d}{dx}(\log|\varphi_{\a\c}'(x)|-\log|\varphi_{\b\d}'(x)|)-\frac{d}{dx}(\log|\varphi_{\c}'(x)|-\log|\varphi_{\d}'(x)|)\right| \leq C\gamma^{-\min\{|\c|,|\d|\}}.
\end{equation}
By our non-triviality assumption, there exists $L\in \mathbb{N}$ and  $\c,\d\in \A^{L}$ such that $\varphi_{\c}(I)\cap\varphi_{\d}(I)=\emptyset$. We let $\alpha_1=\d\b\c\a$ and $\alpha_2=\c\a\d\b$ where $\a,\b\in \A^{M}$ satisfy \eqref{eq:EquiUNI}. We immediately have that $\varphi_{\alpha_1}(I)\cap\varphi_{\alpha_2}(I)=\emptyset$ and $p_{\alpha_1}=p_{\alpha_2}$. Moreover, by \eqref{e:transferring nonlinearity} it follows that for any $M$ sufficiently large we have $$c_1/2\leq \left|\frac{d}{dx}(\log|\varphi_{\alpha_{1}}'(x)|-\log|\varphi_{\alpha_{2}}'(x)|)\right|\leq 2c_2$$ for all $x\in K$. It follows now by a continuity argument that for all $M$ sufficiently large, there exists $\delta_{M}>0$ such that $$c_1/3\leq \left|\frac{d}{dx}(\log|\varphi_{\alpha_{1}}'(x)|-\log|\varphi_{\alpha_{2}}'(x)|)\right|\leq 3c_2$$ for all $x\in \{x:d(x,K)<\delta_{M}\}.$ We can now appeal to \eqref{e:transferring nonlinearity} again to assert that in fact for any $M$ sufficiently large, for any $l\geq 1$ we have 
\begin{equation}
	\label{eq:lUNI}
	c_1/4\leq \left|\frac{d}{dx}(\log|\varphi_{\alpha_{1}^l}'(x)|-\log|\varphi_{\alpha_{2}^l}'(x)|)\right|\leq 4c_2
\end{equation} for all $x\in \{x:d(x,K)<\delta_{M}\}.$. 

Let $N=2(M+L)$. Summarising the above, we have shown that for any $M$ sufficiently large, there exists  $w^*=\{\alpha_1,\alpha_2\}\subset \A^{N}$ such that properties $4a,$ $4b,$ and $4c$ hold. Property $4b$ holds because of \eqref{e:differentUNI} and \eqref{eq:lUNI}. It remains to show that for $M$ sufficiently large we can construct $w_1\ldots w_m\subset \A^{N}\setminus w^*$ so that properties $1, 2$ and $3$ are satisfied.

Let $\mu_{uni}$ be the self-conformal measure corresponding to the uniform probability vector $\p=(\sharp \A^{-1})_{\a\in \A}$. Adapting an argument of Feng and Lau \cite{FL} to the setting of self-conformal measures, there exists $C>0$ and $\alpha>0$ such that 
\begin{equation}
	\label{eq:uniform frostman}
	\mu_{uni}(B(x,r))\leq Cr^{\alpha}
\end{equation}
 for all $x\in \mathbb{R}$. Since our IFS is uniformly contracting, there exists $\gamma>1$ and $C_{1}>0$ such that $$|\varphi_{\a}(I)|\leq C_{1}\gamma^{-N}$$ for any $\a\in \A^{N}$. Using this inequality together with \eqref{eq:uniform frostman} and the fact $p_{\a}=\sharp \A^{-N}$ for all $\a\in \A^{N}$ yields
  \begin{equation}
 	\label{eq:count intersections}
 	\sharp\left\{\a\in \A^{N}: \varphi_{\a}(I)\cap B(x,C_{1}\gamma^{-N})\neq \emptyset\right\}\leq 3CC_{1}^{\alpha}\sharp\A^{N}\gamma^{-N\alpha}
 \end{equation} for any $x\in\mathbb{R}$.

Now let $\{\a_i\}_{i=1}^{\sharp\A^{N}-2}$ be an enumeration of the elements of $\A^{N}\setminus w^*$ such that if $i<j$ then the left endpoint of $\varphi_{\a_i}(I)$ lies to the left of the left endpoint of $\varphi_{\a_j}(I)$. We also let 
\begin{equation}
	\label{eq:T_N def}
	T_{N}:=\lfloor 3CC_{1}^{\alpha}\sharp\A^{N}\gamma^{-N\alpha}\rfloor +1.
\end{equation} The significance of the parameter $T_{N}$ is that if $j\geq i+T_{N}$ then $$\varphi_{\a_i}(I)\cap \varphi_{\a_j}(I)=\emptyset.$$ This fact follows form \eqref{eq:count intersections}.  

For each $0\leq j\leq  \lfloor \frac{\sharp \A^N-2}{2T_{N}}\rfloor-1$ and $1\leq k\leq T_{N}$ we let $$\tilde{w}_{j,k}:=\{\a_{k+2jT_{N}},\a_{k+2jT_{N}+T_{N}}\}.$$ Notice that for any $0\leq j\leq  \lfloor \frac{\sharp \A^N-2}{2T_{N}}\rfloor-1$ and $1\leq k\leq T_{N},$ because the subscripts of $\a_{k+2jT_{N}}$ and $\a_{k+2jT_{N}+T_{N}}$ differ by $T_{N}$ we have
$$\varphi_{\a_{k+2jT_{N}}}(I)\cap\varphi_{\a_{k+2jT_{N}+T_{N}}}(I)=\emptyset.$$
Moreover, $\tilde{w}_{j,k}\cap \tilde{w}_{j',k'}=\emptyset$ for $(j,k)\neq (j',k')$. Our proof is almost complete, it remains to allocate those elements of $\{\a_i\}_{i=2T_{N}\lfloor \frac{\sharp \A^{N}-2}{2T_{N}}\rfloor+1}^{\sharp \A^{N}-2}$ to appropriate subsets of $\A^{N}$. The cardinality of $\{\a_i\}_{i=2T_{N}\lfloor \frac{\sharp \A^{N}-2}{2T_{N}}\rfloor+1}^{\sharp \A^{N}-2}$ is at most $2T_{N}$. Therefore to each $2T_{N}\lfloor \frac{\sharp \A^{N}-2}{2T_{N}}\rfloor+1\leq i\leq \sharp \A^{N}-2$ we can associate a unique pair $(j_i,k_i)$ satisfying $0\leq j\leq 1$ and $1\leq k\leq T_{N}.$ Notice that the largest subscript for a word $\a_i$ contained in $\tilde{w}_{j,k}$ for some $0\leq j\leq 1$ and $1\leq k\leq T_{N}$ is $4T_{N}$. It follows from \eqref{eq:T_N def} that for $M$ sufficiently large, for any $i$ satisfying $2T_{N}\lfloor \frac{\sharp \A^{N}-2}{2T_{N}}\rfloor+1\leq i\leq \sharp \A^{N}-2$ we have $$i-4T_{N}\geq  2T_{N}\lfloor \frac{\sharp \A^{N}-2}{2T_{N}}\rfloor+1-4T_{N}\geq T_{N}.$$ Therefore $$\varphi_{\a_{i}}(I)\cap \varphi_{\a'}(I)=\emptyset$$ for any $i$ satisfying $2T_{N}\lfloor \frac{\sharp \A^{N}-2}{2T_{N}}\rfloor+1\leq i\leq \sharp \A^{N}-2$ and $\a'\in \tilde{w}_{j_i,k_i}$. To each $i$ satisfying $2T_{N}\lfloor \frac{\sharp \A^{N}-2}{2T_{N}}\rfloor+1\leq i\leq \sharp \A^{N}-2$ we associate the subset $\{\a_i\}\cup \tilde{w}_{j_i,k_i}$. Taking $w_1,\ldots,w_m$ to be the subsets $\{\a_i\}\cup \tilde{w}_{j_i,k_i}$ together with the remaining unchanged $\tilde{w}_{j,k}$, we see now that properties $1$, $2$, and $3$ are satisfied. This completes our proof.

\end{proof}
Proposition \ref{prop:IFS partition} now allows us to make a significant simplification in our proof of Theorem \ref{thm:Linftybound}. To prove Theorem \ref{thm:Linftybound} it suffices to show that the same bound holds for $\cL_{s}^{N}$ for some large $N\in \N$. Now using the fact that $\cL_{s}^{N}$ coincides with the transfer operator corresponding to the IFS $\{\varphi_{\a}\}_{\a\in \A^{N}},$ we can apply Proposition \ref{prop:IFS partition} to assert that without loss of generality our original IFS $\{\varphi_{a}\}_{a\in \A}$ is such that there exists $w^*,w_1,\ldots,w_m\subset \A$ satisfying the following properties: 
\begin{enumerate} 
\item $w^*\cup w_1\cup \cdots \cup w_m=\A. $ Moreover, this union is disjoint.

\item  $ \sharp w_i\in \{2,3\}$ for $1\leq i\leq m$.
\item For any $1\leq i\leq m$, for distinct $a,b\in w_i$ we have $$\varphi_{a}(I)\cap \varphi_{b}(I)=\emptyset.$$
\item $w^*=\{\alpha_1,\alpha_2\}$ and these words satisfy:
\begin{itemize}
	\item[a.] $$\varphi_{\alpha_1}(I)\cap \varphi_{\alpha_2}(I)=\emptyset.$$
	\item[b.] There exists $c_1,c_2,\delta>0$ such that for all $x\in \{x:d(x,K)<\delta\}$ and $l\in\mathbb{N}$ we have $$c_1\leq \left|\frac{\varphi_{\alpha_{1}^{l}}''(x)}{\varphi_{\alpha_{1}^{l}}'(x)}-\frac{\varphi_{\alpha_{2}^{l}}''(x)}{\varphi_{\alpha_{2}^{l}}'(x)}\right|\leq c_2.$$
	\item[c.] $p_{\alpha_1}=p_{\alpha_2}.$
\end{itemize} 

\end{enumerate}

We now introduce some more notation to complement this partition of $\A$. We let $$\Omega:=\{w^*,w_1,\ldots,w_m\}\quad \textrm{ and }\quad \Omega^*:=\cup_{n=1}^{\infty}\Omega^n.$$ Moreover, for $w\in \Omega$ and a finite word $w_1\ldots w_n\in \Omega^*$ we let $$q_{w}=\sum_{a\in w}p_a\qquad \text{ and }\qquad q_{w_1\ldots w_n}=\prod_{i=1}^{n}q_{w_i}.$$ For each $w\in \Omega$ and $a\in w$ we let $$p_{a,w}=\frac{p_a}{q_w}.$$ We emphasise that $\sum_{a\in w}p_{a,w}=1$. Moreover, it is a consequence of our assumption that $p_{\alpha_1}=p_{\alpha_2}$ that we also have $$p_{\alpha_1,w^*}=p_{\alpha_2,w^*}=\frac{1}{2}.$$

Now each $w$ is equipped with a probability vector, we can define the associated transfer operators. For $f\in C^1(\mathbb{R})$ and $w\in \Omega$ let $$\cL_{w}(f)=\sum_{a\in w}p_{a,w}f(\varphi_{a}(x)).$$ We similarly define their complex analogues $$\cL_{s,w}(f)=\sum_{a\in w}p_{a,w}|\varphi_{a}'(x)|^{s}f(\varphi_{a}(x)).$$ The operator $\cL_{s,w^*}$ will exhibit interesting behaviour. We emphasise that this operator will take the following simpler form $$\cL_{s,w^*}(f)=\sum_{a\in w^*}\frac{1}{2}|\varphi_{a}'(x)|^{s}f(\varphi_{a}(x)).$$ Our proof of Theorem \ref{thm:Linftybound} relies upon us establishing an appropriate spectral gap result for this operator. The following disintegration lemma will allow us to take this spectral gap information for $\cL_{s,w^*}$ and use it to derive spectral gap information for our original complex transfer operator $\cL_{s}$. 
\begin{lemma}
	\label{lemma:operator disintegration}
For $n\in\mathbb{N}$ and $s\in \mathbb{C}$ we have $$\cL_{s}^{n}=\sum_{w=w_1\ldots w_n\in \Omega^n}q_{w}\cdot \cL_{s,w_n}\circ \cdots \circ \cL_{s,w_1}.$$
\end{lemma}
\begin{proof}
Let $n\in\mathbb{N},$ $s\in \mathbb{C}$ and $f\in C^{1}(\mathbb{R})$. We observe the following:
\begin{align*}
	\cL_{s}^{n}(f)=\sum_{\a\in \A^n}p_{\a}|\varphi_{\a}'(x)|^sf(\varphi_{\a}(x))&=\sum_{w=w_1\ldots w_n\in \Omega^n}\sum_{\stackrel{a_1\ldots a_n\in \A^n}{a_i\in w_i,\, 1\leq i\leq n}}p_{\a}|\varphi_{\a}'(x)|^sf(\varphi_{\a}(x))\\
	& =\sum_{w=w_1\ldots w_n\in \Omega^n}q_{w}\sum_{\stackrel{a_1\ldots a_n\in \A^n}{a_i\in w_i,\, 1\leq i\leq n}}\frac{p_{\a}}{q_w}|\varphi_{\a}'(x)|^sf(\varphi_{\a}(x))\\
	&=\sum_{w=w_1\ldots w_n\in \Omega^n}q_{w}\cdot (\cL_{b,w_n}\circ \cdots \circ \cL_{b,w_1})(f).
\end{align*}	
\end{proof}

We finish this discussion of $\Omega$ and its properties by introducing some notation. We let $K^*$ be the unique non-empty compact set satisfying $$K^*:=\varphi_{\alpha_1}(K^*)\cup \varphi_{\alpha_2}(K^*),$$ i.e. $K^*$ is the self-conformal set for the IFS $\{\varphi_{\alpha_1},\varphi_{\alpha_2}\}$. We also let $\mu_{*}$ be the unique Borel probability measure satisfying $$\mu_{*}=\frac{\mu_{*}\circ \varphi_{\alpha_1}^{-1}}{2}+\frac{\mu_{*}\circ \varphi_{\alpha_2}^{-1}}{2}.$$ For any word $w=w_1\ldots w_n\in \Omega^*$ we let $$K_{w}:=\bigcup_{\stackrel{\a\in \A^n}{a_i\in w_i,\, 1\leq i\leq n}}\varphi_{\a}(K^*)$$ and $$\mu_{w}:=\sum_{\stackrel{\a\in \A^n}{a_i\in w_i,\, 1\leq i\leq n}}p_{\a,w}\cdot \mu_{*}\circ \varphi_{\a}^{-1}.$$ Notice that for any word $w\in\Omega^*$ the measure $\mu_{w}$ is a probability measure supported on $K_{w}$. These measures will play a similar role to that of stationary measures when one considers compositions of a single transfer operator. 

\subsection{A spectral gap for random transfer operators} \label{sec:proofspect}

The purpose of this section is to prove Theorem \ref{thm:Linftybound} whose statement we now recall. 

\begin{thm}
	\label{thm:Linftybound2}
Let $\Phi = \{\phi_a : a \in \A\}$ be a non-trivial uniformly contracting $C^2$ iterated function system satisfying the UNI condition \eqref{eq:UNI}. Then there exists $0 < \rho_0 < 1$ such that for $s = r+ib \in \C$ with $|r|$ sufficiently small and $|b|$ sufficiently large, the operator $\cL_{s}$ satisfies for all $n \in \N$ and $f \in C^1(\R)$:
$$ \|\cL_{s}^n f\|_b \lesssim \rho_0^{n} |b|^{1/2} \|f\|_b.$$
Thus there exists $0 < \delta < 1$ such that for all $|r|$ sufficiently small and $|b|$ sufficiently large, the spectral radius satisfies
$$\rho(\cL_s)\leq 1-\delta.$$
\end{thm}

In the proof of Theorem \ref{thm:Linftybound2}, we end up getting the contraction from the imaginary part $ib$ of $s = r+ib$, and the real part can in the worst case (especially if $r < 0$) cause expansion. To control this, the assumption that $|r|$ is small enough is absorbed by weakening the contraction rate using the uniform contraction of the maps $\phi_a$. Recall that there exists $1 < \gamma < \gamma_1$ such that for all $x \in I$ and $n \in \N$, if $\a\in \A^n$ then
$$\gamma_1^{-n} \lesssim |\varphi_\a'(x)| \lesssim  \gamma^{-n}.$$
Thus for any $r \in \R$ and $n \in \N,$ there exists $c_0>0$ such that we have
\begin{align}\label{eq:unifexp} \sup_{x \in I} \sup_{\a \in \A^n} |\varphi_{\a}'(x)|^r \leq c_0 \gamma_{1}^{n|r|}.\end{align}

The following proposition is the first step towards proving Theorem \ref{thm:Linftybound2}.
\begin{prop}
	\label{prop:normbound}
	There exists $N\in\mathbb{N}$ and $\rho\in(0,1)$ such that if $w_1\ldots w_{N\lfloor n/N\rfloor}$ satisfies
	$$\sharp \{0\leq i\leq \lfloor n/N\rfloor -1:w_{iN+1}\ldots w_{(i+1)N}=(w^*)^N\} \geq cn/N$$
	for some $c > 0$, then for any $s=r+ib$ with small enough $|r|$ and large enough $|b|$ we have
	$$\int_{K_{\tilde w}} |\cL_{s,w_n} \circ  \dots \circ \cL_{s,w_1} (f)|^2  d\mu_{\tilde w} \leq \rho^{cn/N} \|f\|_b^2$$for any word $\tilde{w}$.
\end{prop}
Let $N$ be as in the statement of Proposition \ref{prop:normbound}. We say that a word $w_1\ldots w_n\in \Omega^n$ is good if 
$$\sharp \{0\leq i\leq \lfloor n/2N\rfloor -1:w_{iN+1}\ldots w_{(i+1)N}=(w^*)^N\} \geq \frac{(p_{\alpha_1}+p_{\alpha_2})^Nn}{5N}.$$ Similarly, we say that $w_1\ldots w_n\in \Omega^n$ is bad if it fails to be good. 
The significance of the bound $\frac{(p_{\alpha_1}+p_{\alpha_2})^Nn}{5N}$ is that it is strictly less than $(p_{\alpha_1}+p_{\alpha_2})^N\left \lfloor\frac{n}{2N} \right\rfloor,$ which is the expectation for the number of $0\leq i\leq \lfloor n/2N\rfloor -1$ satisfying $w_{iN+1}\ldots w_{(i+1)N}=(w^*)^N$. Thus we can use large deviation bounds, see for instance Hoeffding \cite{Hoe}, to conclude that there exists $\rho_{2}\in(0,1)$ such that \begin{equation}
	\label{e:bad measure}
	\sum_{\textrm{ bad }w_1\ldots w_n}q_{w}\lesssim \rho_{2}^{n/N}.
\end{equation} This observation together with the following Theorem is what allows us to prove Theorem \ref{thm:Linftybound2}.
\begin{thm}
	\label{thm:randomspectralgap}
There exists $\rho_1\in(0,1)$ such that for $s=r+ib$ with $|r|$ sufficiently small and $|b|$ sufficiently large, for all $n \in \N$ and for all good words $w_1\dots w_n$ and $f \in C^1$, we have
	$$ \|\cL_{s,w_n} \circ \cdots \circ \cL_{s,w_1} (f)\|_\infty \lesssim \rho_1^{n} |b|^{1/2} \|f\|_b.$$
\end{thm}
We now show how Theorem \ref{thm:Linftybound2} follows from Theorem \ref{thm:randomspectralgap}.

\begin{proof}[Proof of Theorem \ref{thm:Linftybound2}]
	
	By Lemma \ref{lemma:operator disintegration}, \eqref{eq:unifexp}, and Theorem \ref{thm:randomspectralgap}, for $|r|$ sufficiently small and $|b|$ sufficiently large, for all $n\in \N$ and $f\in C^{1}(\R)$ we have
	\begin{align*}
		\|\cL_{s}^{n} f\|_{\infty}&\leq \sum_{\text{ good }w_1\ldots w_n}q_w\cdot \| \cL_{s,w_n}\circ \cdots \circ \cL_{s,w_1}(f)\|_{\infty}+\sum_{\text{ bad }w_1\ldots w_n}q_w\cdot \| \cL_{s,w_n}\circ \cdots \circ \cL_{s,w_1}(f)\|_{\infty}\\
		&\lesssim \sum_{\text{ good }w_1\ldots w_n}q_w\rho_{1}^{n}|b|^{1/2}\|f\|_{b}+ \gamma_{1}^{n|r|} \sum_{\text{ bad }w_1\ldots w_n}q_w  \cdot \|f\|_{b}\\
		&\lesssim \rho_{1}^{n}|b|^{1/2}\|f\|_{b} + \gamma_{1}^{n|r|} \rho_{2}^{n/N}\|f\|_{b}\\
		&\lesssim \max\{\rho_1,\gamma_{1}^{|r|} \rho_2^{1/N}\}^n|b|^{1/2}\|f\|_{b}
	\end{align*}
	Here $\rho_{2}$ is as in \eqref{e:bad measure}. For $|r|$ sufficiently small we have that $\max\{\rho_1,\gamma_{1}^{|r|} \rho_2^{1/N}\}\leq \max\{\rho_1,\rho_2^{1/2N}\}.$ Therefore, taking $\rho_{3}:=\max\{\rho_1,\rho_2^{1/2N}\}$ we have $$ \|\cL_{s}^n f\|_\infty \lesssim \rho_3^{n} |b|^{1/2} \|f\|_b$$ for all $|r|$ sufficiently small. 
	
	To get the $\|\cdot\|_b$ bound, we also need to bound the derivative term $\|(\cL_{s}^{n} f)'\|_\infty$. For this purpose, set $m := \lfloor n/2 \rfloor.$ We have for any $x \in I$ that:
	$$|\cL_s^{n} f(x)| = |\cL_{s}^{(n-m)} (\cL_s^m f)(x)|.$$
	Moreover, by uniform expansion, uniform contraction and bounded distortions we have the bound
	 $$\| (\cL_s^m f)' \|_\infty \lesssim |b| \gamma_1^{m|r|} \|f\|_\infty + \gamma^{-m}\gamma_1^{m|r|} \|f'\|_\infty \lesssim |b| \|f\|_b(\gamma_1^{m|r|} + \gamma^{-m}\gamma_1^{m|r|}).$$
	Thus, combining the above with what we proved earlier, we obtain:
		\begin{align*}
		\|(\cL_{s}^{n} f)'\|_\infty & \lesssim |b| \gamma_1^{(n-m)|r|} \| \cL_s^m f\|_\infty + \gamma^{-(n-m)} \gamma_1^{(n-m)|r|} \| (\cL_s^m f)' \|_\infty\\
& \lesssim |b| \gamma_1^{(n-m)|r|}  \rho_3^{m} |b|^{1/2} \|f\|_b + \gamma^{-(n-m)} \gamma_1^{(n-m)|r|} |b| \|f\|_b(\gamma_1^{m|r|} + \gamma^{-m}\gamma_1^{m|r|})\\
& \lesssim \max\{\gamma_1^{|r|/2}  \rho_3^{1/2},\gamma^{-1/2} \gamma_1^{|r|},\gamma^{-1} \gamma_1^{|r|}\}^n |b|^{3/2} \|f\|_b.
	\end{align*}
In the final line we have used that $m=\lfloor n/2 \rfloor.$ For $|r|$ sufficiently small we have the bound $\max\{\gamma_1^{|r|/2}  \rho_3^{1/2},\gamma^{-1/2} \gamma_1^{|r|},\gamma^{-1} \gamma_1^{|r|}\}\leq \max\{  \rho_3^{1/4},\gamma^{-1/4} ,\gamma^{-1/2} \}.$ Therefore, taking $\rho_{4}:=\max\{  \rho_3^{1/4},\gamma^{-1/4} ,\gamma^{-1/2} \}$, we have $$\|(\cL_{s}^{n} f)'\|_\infty\lesssim \rho_{4}^n|b|^{3/2} \|f\|_b$$ for all $r$ sufficiently small. Theorem \ref{thm:Linftybound} now follows by choosing $\rho_0 = \min\{\rho_3,\rho_4\} \in (0,1)$.%
\end{proof}

We will now explain why Theorem \ref{thm:randomspectralgap} follows from Proposition \ref{prop:normbound}.

\begin{proof}[Proof of Theorem \ref{thm:randomspectralgap}] For any $x\in[0,1]$, using the definition of the transfer operator we have:
	$$|\cL_{s,w_n} \circ \dots \circ \cL_{s,w_1}(f)(x)|^2 \leq \cL_{r,w_{n} \circ \dots \circ \cL_{r,w_{N\lfloor n/2N\rfloor +1}}} ( |\cL_{s,w_{N\lfloor n/2N\rfloor}} \circ \dots \cL_{s,w_1}(f)(x)|^2).$$ Therefore, applying Proposition \ref{prop:normbound}, we have the following for $|r|$ sufficiently small and $|b|$ sufficiently large:
	\begin{align*}
		&\cL_{r,w_n} \circ \dots \circ \cL_{r,w_{N\lfloor n/2N\rfloor+1}} ( |\cL_{s,w_{N\lfloor n/2N\rfloor}} \circ \dots \circ \cL_{s,w_1}(f)(x)|^2)\\
		&\lesssim \gamma_{1}^{|r|n/2} \cL_{w_n} \circ \dots \circ \cL_{N\lfloor n/2N\rfloor+1} ( |\cL_{s,w_{N\lfloor n/2N\rfloor}} \circ \dots \circ \cL_{s,w_1}(f)(x)|^2)\\
		&=\gamma_{1}^{|r|n/2} \int_{K_{w_{N\lfloor n/2N\rfloor+1}\dots w_{n}}} |\cL_{s,w_{N\lfloor n/2N\rfloor}} \circ \dots \circ \cL_{s,w_1}(f)(x)|^2 \, d\mu_{w_{N\lfloor n/2N\rfloor+1} \dots w_{n}} \\
		& \quad + O(\gamma_{1}^{|r|n/2}  \gamma^{-n/2} \|(|\cL_{s,w_{N\lfloor n/2N\rfloor }} \circ \dots \circ \cL_{s,w_1}(f)|^2)'\|_{\infty})\\
		& \leq \gamma_{1}^{|r|n/2} \rho^{cn/N}\|f\|_{b}^{2}+ O(\gamma_{1}^{|r|n/2}  \gamma^{-n/2} \|(|\cL_{s,w_{N\lfloor n/2N\rfloor}} \circ \dots \circ \cL_{s,w_1}(f)|^2)'\|_{\infty}).
	\end{align*}
	In the above we have taken $$c=\frac{(p_{\alpha_1}+p_{\alpha_2})^N}{5}$$ and used the fact that $w_1\ldots w_n$ is good.
	Furthermore, we always have the bound
	$$\|(|\cL_{s,w_{N\lfloor n/2N\rfloor}} \circ \dots \circ \cL_{s,w_1}(f)|^2)'\|_{\infty} \lesssim \gamma^{|r|n/2}|b| \|f\|_{b}^2.$$ Applying this bound in the above, we see that
	$$|\cL_{s,w_n} \circ \dots \circ \cL_{s,w_1}(f)(x)|^2\lesssim \max\{\gamma_{1}^{|r|} \rho^{c/N},\gamma_{1}^{|r|}\gamma^{-1/2}\}^n|b| \|f\|_{b}^2$$ for all $x\in I$. For $|r|$ sufficiently small we have $\max\{\gamma_{1}^{|r|} \rho^{c/N},\gamma_{1}^{|r|}\gamma^{-1/2}\}\leq \max\{ \rho^{c/2N},\gamma^{-1/4}\}$. Therefore, taking $\rho_{1}=\max\{ \rho^{c/2N},\gamma^{-1/4}\}^{1/2}$, we have $$\|\cL_{s,w_n} \circ \dots \circ \cL_{s,w_1}(f)\|_{\infty}\lesssim \rho_{1}^n|b|^{1/2} \|f\|_{b}$$ for all $|r|$ sufficiently small and $|b|$ sufficiently large. This completes our proof.
	
	
\end{proof}

The missing piece in our argument is a proof of Proposition \ref{prop:normbound}. We do this in the next section by constructing suitable Dolgopyat \cite{Dolgopyat} type random operators.

\subsection{Reduction to Dolgopyat type random operators} 

Our purpose now is to show how Proposition \ref{prop:normbound} follows from the following crucial lemma. This lemma gives a construction of certain \textit{random Dolgopyat operators} (which we define formally later in its proof).

\begin{lemma}[Construction of random Dolgopyat operators]
	\label{lma:dolgopyat}
	There exists $N \in \N,$ $A > 1$ and $\rho = \rho(w^*) \in (0,1),$ such that for all $s=r+ib$ with $|r|$ sufficiently small and $|b|$ sufficiently large, for any $w' \in \bigcup_{n=1}^\infty \Omega^n$ there exists a finite set of bounded operators $(\cN_s^J)_{J \in \cE_s}$ on $C^1(I)$ satisfying the following properties:
	\begin{itemize}
		\item[(1)] The cone 
		$$C_{A|b|} = \{f \in C^1(I) : f > 0, |f'(x)| \leq A|b|f(x)\}$$
		is stable under $\cN_s^J$ for all $J \in \cE_s$, that is, if $H \in C_{A|b|}$ and $J \in \cE_s$, then 
		$$|\cN_s^J(H)'(x)| \leq A|b| \cN_s^J(H)(x)$$ 
		for all $x \in I$.
		\item[(2)] For all $H \in C_{A|b|}$ and $J \in \cE_s$, 
		$$\int_{K_{w'}} |\cN_s^J H|^2 \, d\mu_{w'} \leq \rho \int_{K_{(w^*)^N w'}} |H|^2 \, d\mu_{(w^*)^N w'}$$
		\item[(3)] Given $H \in C_{A|b|}$ and $f \in C^1(I)$  such that $|f| \leq H$ and $|f'| \leq A|b|H$, there exists $J \in \cE_b$ such that
		$$|\cL_{s,w^*}^N f| \leq \cN_s^J (H)$$
		and
		$$|(\cL_{s,w^*}^N f)'| \leq A|b|\cN_s^J (H)$$
	\end{itemize}
\end{lemma}

Assuming Lemma \ref{lma:dolgopyat} we now focus on proving Proposition \ref{prop:normbound}. Our proof of this proposition depends upon a careful choice of random Dolgopyat operators. Our analysis naturally falls into two cases, whether we observe the word $(w^*)^{N}$ or not. 

\begin{proof}[Proof of Proposition \ref{prop:normbound}]
	Let $N$ be as in Lemma \ref{lma:dolgopyat} and let $w_1\ldots w_{N\lfloor n/N\rfloor}$ satisfy $$\sharp \{0\leq i\leq \lfloor n/N\rfloor -1:w_{iN+1}\ldots w_{(i+1)N}=(w^*)^N\} \geq cn/N$$
	for some $c > 0$. Let $s=r+ib$ be such that $|r|$ is sufficiently small and $|b|$ is sufficiently large so that Lemma \ref{lma:dolgopyat} applies.
	
	We inductively choose a sequence of operators $\tilde \cN_{0},\ldots \tilde \cN_{\lfloor n/N\rfloor -1}$ as follows:
	\begin{itemize}
		\item[(1)] Consider the block $w_1\ldots w_N$. If $w_1\ldots w_N\neq (w^*)^{N}$ then we let 
		$$\tilde \cN_{0}=\cL_{r,w_N}\circ \cdots \circ \cL_{r,w_1}.$$ 
		If $w_1\ldots w_N= (w^*)^{N}$ then we let 
		$$\tilde \cN_{0} :=\cN_{s}^{J},$$ 
		where $\cN_{s}^{J}$ is the random Dolgopyat operator coming from Lemma \ref{lma:dolgopyat} for the choice of the word $$w' = w_{N+1}\ldots w_{N\lfloor n/N\rfloor}\tilde{w}$$ and where $H$ is the \textit{constant} function 
		$$H:= H_0 = \|f\|_{b}\1.$$
		\item[(2)] Assume we have made choices of $\tilde \cN_{k}$ for $0\leq k\leq \ell-1$. If $w_{r,\ell N+1}\ldots w_{r,(\ell +1)N}\neq  (w^*)^{N}$ then we let 
		$$\tilde \cN_{\ell} :=\cL_{r,w_{(\ell+1)N}}\circ \cdots \circ \cL_{r,w_{\ell N+1}}.$$ 
		If $w_{\ell N+1}\ldots w_{(\ell+1)N}= (w^*)^{N}$, then we take our operator to be the random Dolgopyat operator given by Lemma \ref{lma:dolgopyat} with the choice of the word $$w':=w_{(\ell+1)N+1}\ldots w_{N\lfloor n/N\rfloor}\tilde{w}$$ and the choice of the function
		$$H:=H_{\ell}=\tilde\cN_{\ell-1} \circ \dots \circ \tilde \cN_{0} (\|f\|_b \1).$$  
	\end{itemize}
	We repeat this process until we have defined  $\tilde \cN_{0},\ldots \tilde \cN_{\lfloor n/N\rfloor -1}.$

	These operators will allow us to bound the $L^2$ norms of our random operators as in the statement of Proposition \ref{prop:normbound}. The first step towards achieving this bound is to observe the following inequality:
	\begin{equation}
		\label{eq:relateddolgopyat}|\cL_{s,w_N\lfloor n/N\rfloor} \circ  \dots \circ \cL_{s,w_1} (f)|  \leq   \tilde\cN_{\lfloor n/N\rfloor-1} \circ \dots \circ \tilde \cN_0(H).
	\end{equation}
		We omit the proof of \eqref{eq:relateddolgopyat}. Its proof relies upon a simple inductive argument that makes use of property $3$ of Lemma \ref{lma:dolgopyat}.

		By \eqref{eq:relateddolgopyat}, it is enough to bound the integral
		$$ \int_{K_{\tilde w}}| \tilde\cN_{\lfloor n/N\rfloor-1} \circ \dots \circ \tilde \cN_0(H)|^2 \, d\mu_{\tilde w}.$$
		Let
		$$D = \{0 \leq \ell \leq \lfloor n/N\rfloor -1 : \tilde \cN_{\ell} = \cN_{s}^J\},$$
		that is, $D$ is the set of subscripts where $\tilde \cN_{\ell}$ has been chosen to be an operator coming from Lemma \ref{lma:dolgopyat}. Now the idea is that along the blocks corresponding to $\ell \in D$, we will see decay due to Lemma \ref{lma:dolgopyat}, and for the other blocks we can control the expansion using \eqref{eq:unifexp}. 
		
		Let us look at the last block of length $N$. If $\lfloor n/N\rfloor -1 \notin D$, then by the Cauchy-Schwartz inequality, \eqref{eq:unifexp} and the definition of the unperturbed transfer operator, we have the bound
		\begin{align*}
			&\int_{K_{\tilde w}}| \tilde\cN_{\lfloor n/N\rfloor-1} \circ \dots \circ \tilde \cN_0(H)|^2 \, d\mu_{\tilde w}\\
			\leq &\gamma_{1}^{2|r|N} \int_{K_{w_{N\lfloor n/N\rfloor-N+1} \dots w_{N\lfloor n/N\rfloor} \tilde w}} | \tilde\cN_{\lfloor n/N\rfloor-2} \circ \dots \circ \tilde \cN_0(H)|^2 \, d\mu_{w_{N\lfloor n/N\rfloor-N+1} \dots w_{N\lfloor n/N\rfloor} \tilde w}. 
		\end{align*}If $\lfloor n/N\rfloor -1 \in D$, we bound it instead by
		$$\rho \int_{K_{w_{N\lfloor n/N\rfloor-N+1} \dots w_{N\lfloor n/N\rfloor} \tilde w}} | \tilde\cN_{\lfloor n/N\rfloor-2} \circ \dots \circ \tilde \cN_0(H)|^2 \, d\mu_{w_{N\lfloor n/N\rfloor-N+1} \dots w_{N\lfloor n/N\rfloor} \tilde w},$$
		which is possible by property 2 from Lemma \ref{lma:dolgopyat}. We then continue this process at the next stage. We bound 
		$$\int_{K_{w_{N\lfloor n/N\rfloor-N+1} \dots w_{N\lfloor n/N\rfloor} \tilde w}} | \tilde\cN_{\lfloor n/N\rfloor-2} \circ \dots \circ \tilde \cN_0(H)|^2 \, d\mu_{w_{N\lfloor n/N\rfloor-N+1} \dots w_{N\lfloor n/N\rfloor} \tilde w}$$
		when $\lfloor n/N\rfloor -2 \notin D$ by
		$$\gamma_{1}^{2|r|N}\int_{K_{w_{N\lfloor n/N\rfloor-2N+1}\dots  w_{N\lfloor n/N\rfloor} \tilde w}} | \tilde\cN_{\lfloor n/N\rfloor-3} \circ \dots \circ \tilde \cN_0(H)|^2 \, d\mu_{w_{N\lfloor n/N\rfloor-2N+1}\dots  w_{N\lfloor n/N\rfloor}\tilde w},$$
		and if $\lfloor n/N\rfloor -2 \in D$ we use Lemma \ref{lma:dolgopyat} to bound by
		$$\rho \int_{K_{w_{N\lfloor n/N\rfloor-2N+1} \dots w_{N\lfloor n/N\rfloor} \tilde w}} | \tilde\cN_{\lfloor n/N\rfloor-3} \circ \dots \circ \tilde \cN_0(H)|^2 \, d\mu_{w_{N\lfloor n/N\rfloor-2N+1} \dots w_{N\lfloor n/N\rfloor} \tilde w}.$$
		We repeat this process until we have exhausted all of our operators $\tilde \cN_{0},\ldots, \tilde \cN_{\lfloor n/N\rfloor -1}$. Importantly, we will see a $\rho$ contraction every time $\ell \in D$, and a $\gamma_{1}^{2|r|N} $ expansion when $\ell \notin D$. At the same time, $\sharp D \geq c n/N$, so we arrive to
		$$ \int_{K_{\tilde w}}| \tilde \cN_{\lfloor n/N\rfloor -1} \circ \dots \circ \tilde \cN_{0}(H)|^2 \, d\mu_{\tilde w} \leq \gamma_{1}^{2|r|N\lfloor n/N\rfloor} \rho^{cn/N} \int_{K_{w_{1} \dots w_{\lfloor n/N\rfloor} w'}} |H|^2 \, d\mu_{w_{1} \dots w_{\lfloor n/N\rfloor}w'}$$
		and as we chose $H = H_0 = \|f\|_b \1$ and $\mu_{w_{1}\dots w_n w'}$ is a probability measure, we have
		$$ \int_{K_{\tilde w}}| \tilde \cN_{\lfloor n/N\rfloor -1} \circ \dots \circ \tilde \cN_{0}(H)|^2 \, d\mu_{\tilde w} \leq \gamma_{1}^{2|r|N\lfloor n/N\rfloor} \rho^{cn/N} \|f\|_b^2.$$
		Taking $|r|$ sufficiently small that $\gamma_{1}^{2|r|N}<\rho^{-c/2}$ we obtain $$ \int_{K_{\tilde w}}| \tilde \cN_{\lfloor n/N\rfloor -1} \circ \dots \circ \tilde \cN_{0}(H)|^2 \, d\mu_{\tilde w} \leq  \rho^{cn/2N} \|f\|_b^2.$$ This completes our proof.
	\end{proof}

	Thus we are just left with proving Lemma \ref{lma:dolgopyat} by constructing the Dolgopyat-type operators.

		\subsection{Construction of the Dolgopyat operators (proof of Lemma \ref{lma:dolgopyat})}
		
		The starting point for the construction of the Dolgopyat operators is to build a kind of ``tree structure'' using the Cantor sets $K_w$ and various parameters, which eventually will depend on the probability vector $(p_a)_{\a\in \A}$, the IFS, the partition of our IFS and the imaginary part $b$ of the complex number $s$.
		
		\begin{prop}[Tree structure]
			\label{prop:Interval covers}
			There exists constants $A_1',A_1>0$ and $A_2>0$ such that for all $\epsilon$ sufficiently small, for any $w\in \cup_{n=1}^{\infty}\Omega^n$ there exists a finite collection $(V_j)_{1\leq j\leq q}$ of closed intervals ordered from left to right such that:
			\begin{enumerate}
				\item $I\subseteq \bigcup_j V_j$, and 
				$$\mathrm{int}\,V_i\cap \mathrm{int}\,V_j=\emptyset$$ 
				for $i\neq j$. 
				\item For all $1\leq j\leq q$, we have 
				$$\epsilon A_1'\leq |V_j|\leq \epsilon A_1.$$
				\item For all $1\leq j\leq q$ such that $V_j\cap K_{w}\neq \emptyset$, either 
				$$V_{j-1}\cap K_{w}\neq\emptyset \quad \text{and} \quad V_{j+1}\cap K_{w}\neq\emptyset;$$ 
				or 
				$$V_{j-2}\cap K\neq\emptyset\quad \text{and} \quad V_{j-1}\cap K_{w}\neq\emptyset;$$ 
				or 
				$$V_{j+1}\cap K_{w}\neq\emptyset\quad \text{and} \quad V_{j+2}\cap K_{w}\neq\emptyset.$$
				\item For all $1\leq j\leq q$ such that $V_j\cap K_{w}\neq \emptyset$ we have 
				$$\dist(\partial V_j,K_{w})\geq A_2|V_j|.$$
			\end{enumerate} 
		\end{prop}
		
		\begin{proof}
			Fix $w\in \cup_{n=1}^{\infty}\Omega^n$. We define $\pi_{w}:w\times (w^*)^{\N}\to K_{w}$ according to the rule 
			$$\pi_{w}(\a)=\lim_{n\to\infty}\phi_{a_1}\circ \cdots \circ \phi_{a_n}(0).$$ 
			Recall that our separation assumptions means that for any $w\in \Omega$, for distinct $a,b\in w$ we have $\varphi_{a}(I)\cap \varphi_{b}(I)=\emptyset.$ This fact implies that $\pi_{w}$ is a continuous bijection from $w\times (w^*)^{\N}$ to $K_{w}$. 
			
			For any $\epsilon>0$ sufficiently small, for $\a \in w\times (w^*)^{\N},$ we define the $\epsilon$-cutoff of $\a$ to be the unique prefix of $\a$ satisfying the following:
			$$\diam(\phi_{a_1\ldots a_{M}}(I))<\epsilon\qquad \textrm{ and }\qquad \diam(\phi_{a_1 \ldots a_{M-1}}(I))\geq \epsilon.$$
			We let $\Sigma_{\epsilon}$ denote the set of $\epsilon$-cutoff words. To each $\a\in \Sigma_{\epsilon}$ we associate the w-cylinder set
			$$[\a]_w:=\left\{\b\in w\times (w^*)^{\mathbb{N}}:b_1\ldots b_{|\a|}=\a\right\},$$
			and let
			$$K_{\a,w}:=\pi_{w}([\a]_{w}).$$ We have $\cup_{\a\in \Sigma_{\epsilon}}K_{\a,w}=K_{w}.$ The sets $K_{\a,w}$ will be the tools we use to construct the intervals $(V_j)$. Our first step is to derive a separation bound for these sets. 
			
			Let $\a,\a'\in \Sigma_{\epsilon}$ be distinct and $|\a\wedge \a'|=\inf\{k:a_k\neq a_k'\}$. Then \begin{align*}d(K_{\a,w}, K_{\a',w})&\geq d(\phi_{a_1\ldots a_{|\a\wedge \a'|}}(I),\phi_{a_1'\ldots a'_{|\a\wedge \a'|}}(I))\\
				&\geq \inf_{x\in I}\{|\phi_{a_1\ldots a_{|\a\wedge \a'|-1}}'(x)|\}\cdot d(\phi_{a_{|\a\wedge \a'|}}(I),\phi_{a'_{|\a\wedge \a'|}}(I)).
			\end{align*}
			$d(\phi_{a_{|\a\wedge \a'|}}(I),\phi_{a'_{|\a\wedge \a'|}}(I))$ is bounded below by a constant that only depends upon the partition of our IFS. Moreover, by a bounded distortion argument, we know that  $$\inf_{x\in I}|\phi_{a_1\ldots a_{|\a\wedge \a'|-1}}'(x)|\sim \diam(\phi_{a_1 \ldots a_{|\a\wedge \a'|-1}}(I)).$$ Since $a_1\ldots a_{|\a\wedge \a'|-1}$ is a prefix of an $\epsilon$-cutoff word, we know that $\diam(\phi_{a_1\ldots a_{|\a\wedge \a'|-1}}(I))\geq \epsilon.$ It therefore follows from the above that 
			\begin{equation}
				\label{eq:K separation}
				d(K_{\a,w}, K_{\a',w})\gtrsim\epsilon.
			\end{equation} Where the underlying constant depends only upon the partition of our IFS. 
			
			For each word $\a\in \Sigma_{\epsilon},$ by definition there exists $w_{\a,1},w_{\a,2}\in \Omega$ such that  $[\a]_w=\cup_{\b\in w_{a,1}\times w_{a,2}}[\a\b]_w.$ Moreover the following properties holds for each $\a\in \Sigma_{\epsilon}$: \begin{enumerate}
				\item $K_{\a,w}=\cup_{\b\in  w_{\a,1}\times w_{\a,2}}K_{\a\b,w}.$
				\item $\diam(K_{\a\b,w})\sim \epsilon\textrm{ for each }\b\in w_{\a,1}\times w_{\a,2}.$
				\item $d(K_{\a\b,w},K_{\a\b',w})\sim \epsilon$ for distinct $\b,\b'\in w_{\a,1}\times w_{\a,2}.$
				\item $\sharp w_{\a,1}\times w_{\a,2}\geq 4.$
			\end{enumerate}
			Crucially the underlying constants in the above items only depend upon the partition of our IFS. Item $1$ holds by definition. Item $2$ follows from the fact $\a$ is an $\epsilon$-cutoff word and $\diam(K_{\a\b,w})\sim \diam(K_{\a,w})$. The implicit lower bound in item $3$ follows from the same reasoning as that given above to show $d(K_{\a,w}, K_{\a',w})\gtrsim\epsilon$. The implicit upper bound follows since $K_{\a\b,w},K_{\a\b',w}\subset K_{\a,w}$ and $\diam(K_{\a,w})\sim \epsilon$ for any $\epsilon$-cutoff word. The final bound follows because each element of $\Omega$ is a set containing either two or three elements by Proposition \ref{prop:IFS partition}. 
			
			We now use the sets $\{K_{\a\b,w}\}_{\a\in \Sigma_{\epsilon},\b\in w_{\a,1}\times w_{\a,2}}$ to construct the intervals $(V_j)$. It follows from \eqref{eq:K separation}, item $1$, and item $3$ that 
			\begin{equation}
				\label{eq:deeper separation}
				d(K_{\a\b,w},K_{\a'\b',w})\succeq \epsilon
			\end{equation} when $\a,\a'\in \Sigma_{\epsilon}$ are distinct or $\b, \b'\in w_{\a,1}\times w_{\a,2}$ are distinct. Now using \eqref{eq:deeper separation} and item $2$, we can associate to each $K_{\a\b,w}$ a closed interval $V_{\a\b,w}$ so that the following properties are satisfied
			\begin{itemize}
				\item[a.] $K_{\a\b,w}\subset V_{\a\b,w}$ for each $\a\in \Sigma_{\epsilon}$ and $\b\in w_{\a,1}\times w_{\a,2}.$
				\item[b.] $\mathrm{int}\,V_{\a\b,w}\cap \mathrm{int}\,V_{\a'\b',w}=\emptyset$ when $\a,\a'\in \Sigma_{\epsilon}$ are distinct or $\b, \b'\in w_{\a,1}\times w_{\a,2}$ are distinct.
				\item[c.] $d(K_{w},\partial V_{\a\b,w})\gtrsim \epsilon$ for each $\a\in \Sigma_{\epsilon}$ and $\b\in w_{\a,1}\times w_{\a,2}.$
				\item[d.] $\diam(V_{\a\b,w})\sim \epsilon$ for each $\a\in \Sigma_{\epsilon}$ and $\b\in w_{\a,1}\times w_{\a,2}$. 
				\item[e.] If $\a,\a'\in \Sigma_{\epsilon}$ are distinct then $d(V_{\a\b,w},V_{\a'\b',w})\gtrsim\epsilon$ for all $\b\in w_{\a,1}\times w_{\a,2}$ and $\b'\in w_{\a',1}\times w_{\a',2}$.
				\item[f.] For a fixed $\a\in \Sigma_{\epsilon}$, successive $V_{\a\b,w}$ share a common endpoint. 
			\end{itemize}
			We emphasise that each of the implicit constants in the above only depend upon the partition of our IFS. The intervals $\{V_{\a\b,w}\}$ satisfy properties $2,3$ and $4$ of our proposition. Properties $2$ and $4$ follow from items c and d. Property $3$ follows from items $4$ and f.  It remains to address property $1$. By item a, for each $\a\in \Sigma_{\epsilon}$ we have the inclusion $Conv(K_{\a,w})\subset\cup_{\b\in w_{\a,1}\times w_{\a,2}} V_{\a\b,w}.$ Moreover $\{V_{\a\b,w}\}$ satisfies the second part of property $1$ by item b. It suffices therefore to introduce additional closed intervals to fill the gaps between the sets $\cup_{\b\in w_{\a,1}\times w_{\a,2}} V_{\a\b,w}$ so that the first part of property $1$ is satisfied, and so that the second part of this property and property $2$ still hold.  By item e we have $d(\cup V_{\a\b},\cup V_{\a'\b'})\gtrsim \epsilon$ for distinct $\a,\a'\in \Sigma_{\epsilon}$. Therefore we can introduce finitely many closed intervals $\{V_l\}$ each satisfying $\diam(V_l)\sim \epsilon,$ whose union fills the gaps between the sets $\cup_{\b\in w_{\a,1}\times w_{\a,2}} V_{\a\b,w}.$ Moreover, we can insist that successive elements of $\{V_l\}$ only intersect at their endpoints if at all. Taking $\{V_j\}=\{V_l\}\cup\{V_{\a\b,w}\}$ we see that this collection satisfies property $1$ and properties $2,3$, and $4$ still hold. 
		\end{proof}

		Let $w'\in \cup_{n=1}^{\infty}\Omega^n$ be fixed. We now also fix a collection of intervals $(V_j)$ so that Proposition \ref{prop:Interval covers} is satisfied for $\epsilon=\frac{\epsilon'}{|b|}$ with $\epsilon'$ and $\frac{1}{|b|}$ both sufficiently small. For $i\in\{1,2\}$ and $1\leq j\leq q$ we let $Z^{i}_{j}=\varphi_{\alpha_{i}^{N}}(V_j)$.
		
		Properties $2$ and $4$ in Proposition \ref{prop:Interval covers} imply $\dist(K_{w}\cap V_j,\partial V_j)\geq A_2A_1'\frac{\epsilon'}{b}$ whenever $K_{w}\cap V_j\neq \emptyset$. Hence, for all $j$ such that $K_{w}\cap V_j\neq \emptyset,$ there exists a $C^1$ cut off function $\chi_j$ on $I$ such that $0\leq \chi_j\leq 1$, $\chi_j\equiv 1$ on the convex hull of $K_{w}\cap V_j$, and $\chi_j\equiv 0$ outside of $V_j$. Moreover, we can assume 
		\begin{equation}
			\label{cut off derivative}
			\|\chi_{j}'\|_{\infty}\leq A_{3}\frac{|b|}{\epsilon'}
		\end{equation}
		for some constant $A_3$ depending upon the preceding constants. Given $s = r+ib$, the set $\mathcal{J}_{s}$ is defined by $$\mathcal{J}_{s}=\left\{(i,j):i=1,2\textrm{ and }1\leq j\leq q \textrm{ with }V_{j}\cap K_{w}\neq \emptyset\right\}.$$
		Note that by construction $\mathcal{J}_s$ actually only depends on $b$, but the random Dolgopyat operators $\cN_s^J$ are defined using the real part $r$ as well:
		
		\begin{definition}[Random Dolgopyat operators $\cN_s^J$] Let $s = r+ib$ for $|b|$ sufficiently large. Fix $\theta\in(0,1)$ which we will eventually pick to be sufficiently small. Given non-empty $J\subset\mathcal{J}_{s}$, define a function $\chi_{J}\in C^{1}(I)$ by 
			$$
			\chi_{J}:=
			\begin{cases}
				1-\theta\chi_j((\varphi_{\alpha_{i}^{N}}^{-1}x))\, &\textrm{ if }x\in Z_{j}^{i} \textrm{ for }(i,j)\in J\\
				1&\textrm{ otherwise}.
			\end{cases}
			$$
			The \textit{random Dolgopyat operator} $\NN_{s}^{J}$ is defined on $C^{1}(I)$ by 
			$$\NN_{s}^{J}(f):=\cL_{r,w^*}^{N}(\chi_{J}f).$$ 
		\end{definition}
		
		We now set out to prove that these operators satisfy the properties given in Lemma \ref{lma:dolgopyat}. We will begin with property 1 of this lemma.
		
		\begin{proof}[Proof of property 1 of Lemma \ref{lma:dolgopyat}]We start by showing that for suitable constants $A, N$ and $\theta$ the cone $C_{A|b|}$ is stable under $\NN_{s}^{J}.$ 
			
			Given $H\in C_{A|b|}$, assuming $|r| < 1$, for all $x\in I$ we have
			\begin{align*}
				|\NN_{s}^{J}(H)'(x)|&=|\cL_{r,w^*}^{N}(\chi_J H)'(x)|\\
				& \leq \sum_{\a\in (w^*)^N} \frac{1}{2^N} \Big| \frac{\phi_{\a}''(x)}{\phi_{\a}'(x)}\Big|  |\phi_{\a}'(x)|^{r}\chi_{J}(\varphi_{\a}(x))H(\varphi_{\a}(x)) \\	
				&\qquad+ \frac{1}{2^N}|\phi_{\a}'(x)|^{r}\Big(|(\chi_{J}\circ \varphi_{\a})'(x)H(\varphi_{\a}(x))|
				+|\chi_{J}(\varphi_{\a}(x))((H\circ \varphi_{\a})'(x))|\Big).
			\end{align*}
			By a bounded distortion argument, there exists $C_0 > 0$ such that for all $N\in\mathbb{N},$ $x \in I$ and $\a\in (w^*)^N,$ we have$$\Big| \frac{\phi_{\a}''(x)}{\phi_{\a}'(x)}\Big|\leq C_0.$$
			If $(\chi_{J}\circ \varphi_{\a})'(x)\neq 0$, then there exists $(i,j)\in J$ such that $\a=\alpha_{i}^{N}$ and 
			$$\chi_{J}\circ \varphi_{\a}=1-\theta(\chi_{j}\circ \varphi_{\alpha_{i}^{N}}^{-1}\circ \varphi_{\alpha_{i}^{N}}).$$ 
			Differentiating this latter expression and using \eqref{cut off derivative} we obtain 
			$$|(\chi_{J}\circ \varphi_{\a})'| \leq \theta A_{3}\frac{|b|}{\epsilon'}.$$ 
			Moreover, 
			$$|(H\circ \varphi_{\a})'(x)|=|H'(\varphi_{\a}(x))\varphi_{\a}'(x)|\leq A|b|\gamma^{-N}H(\varphi_{\a}(x)),$$ 
			where in the last inequality we used that $H\in C_{A|b|}.$ 
			
			Combining the inequalities above and using that $H>0$, we have
			$$|\NN_{s}^{J}(H)'(x)|\leq C_0\NN_{s}^{J}(H)(x)+A_{3}\theta \frac{|b|}{\epsilon'}\cL_{r,w^*}^{N}(H)(x)+ A|b|\gamma^{-N}\NN_{s}^{J}(H)(x).$$ Note that $H=\chi_{J}H/\chi_{J}\leq \frac{1}{1-\theta}\chi_{J}H.$ Using this inequality, we see that for all $A>2C_0+4A_{3}$, $\theta<\min\{\epsilon',1/2\}$, and $N$ sufficiently large so that $\gamma^{-N}<1/2,$ we have 
			$$|\NN_{s}^{J}(H)'(x)|\leq \left(C_0+\frac{A_3\theta}{(1-\theta)\epsilon'}+A\gamma^{-N}\right)|b|\NN_{s}^{J}(H)(x)\leq A|b|\NN_{s}^{J}(H)(x).$$  
			So the cone $C_{A|b|}$ is stable under $\NN_{s}^{J}$. We have therefore established property $1$ from Lemma \ref{lma:dolgopyat}.
		\end{proof}
		
		We now draw our attention to property $2$ from Lemma \ref{lma:dolgopyat}. 
		
		\begin{definition}[Dense subset]
			We say that $J\subset\mathcal{J}_{s}$ is \textit{dense} if for all $1\leq j\leq q$ such that $V_{j}\cap K_{\omega}\neq\emptyset$, there exists $1\leq j'\leq q$ with $(i,j')\in J$  for some $i\in\{1,2\}$ such that $|j'-j|\leq 2$. 
		\end{definition}
		
		Let $J$ be a dense subset, we denote by $W_{J}$ the subset of $K_w$ defined by $$W_{J}=\{x\in K_w:\exists (i,j)\in J:x\in V_j\}.$$
		
		The following uniform doubling property of the random measures $\mu_w$ will prove to be useful.
		
		\begin{lemma}
			\label{lma:doubling}
			There exists $C>0$ such that for any $w\in \cup_{n=1}^{\infty}\Omega^n$ we have $$\mu_{w}(B(x,2R))\leq C\mu_{w}(B(x,R))$$ for all $x\in K_{\omega}$ and $R>0$. 
		\end{lemma}

		\begin{proof}
			Let $x\in K_{w}$ and $R>0$. Recalling the notation used in the proof of Proposition \ref{prop:Interval covers}, we let $a_1\ldots a_n$ be the unique shortest word such that $K_{a_1\ldots a_n,w}\subset B(x,R)$ and $x\in K_{a_1\ldots a_n,w}$. Then $\mu_{\omega}(B(x,R))\geq \mu_{\omega}(K_{a_1\ldots a_n,w})$. By Proposition \ref{prop:IFS partition} we know that for any $w\in \Omega$, for any distinct $a,b\in w$ we $\varphi_{a}(I)\cap \varphi_{b}(I)=\emptyset$. It follows from this separation property that there exists $l\in \mathbb{N}$ depending only upon the partition of our IFS, such that $B(x,2R)\cap K_{\omega}\subset K_{a_1\ldots a_{n-l},w}$. Therefore $\mu_{w}(B(x,2R))\leq \mu_{w}(K_{a_1\ldots a_{n-l}}).$ Combining this bound with our previous inequality yields $$\frac{\mu_{w}(B(x,2R))}{\mu(B(x,R))}\leq \frac{\mu_{w}(K_{a_1\ldots a_{n-l},w})}{\mu_{w}(K_{a_1\ldots a_n,w})}.$$ Crucially this latter term can be bounded above by a constant that only depends upon the partition of our IFS and the underlying probability vector $\p$. This completes our proof. 
		\end{proof}

		\begin{lemma}
			\label{W_J lemma}
			Let $J$ be a dense subset and $H\in C_{A|b|}$. Then there exists a constant $\tilde{\epsilon}>0$ depending upon $\epsilon',$ the doubling constant from Lemma \ref{lma:doubling}, and the partition of our IFS such that $$\int_{W_J} H\, d\mu_{w}\geq \tilde{\epsilon}\int_{K_{w}}H\, d\mu_{w}.$$
		\end{lemma}
		
		\begin{proof}
			
			Let $\mathcal{G}$ denote the indices in $\{1,\ldots,q\}$ such that $V_k\cap K_{w}\neq\emptyset$. Given $k\in\mathcal{G}$, by the density of $J$ there exists an index $j(k)$ with $(i,j(k))\in J$ for some $i\in\{1,2\}$ such that $|j(k)-k|\leq 2$. By choosing such a $j(k)$ for all $k\in\mathcal{G}$ we get an map $j:\mathcal{G}\to \{1,\ldots,q\}$. Notice that for all $j'\in \{1,\ldots,q\}$ the set $j^{-1}(j')$ contains at most $5$ elements.

			For all $k\in\mathcal{G}$, we choose an arbitrary $u_k\in K_{w}\cap V_k$. We have $V_{j(k)}\subset B(u_k,R)$ and $V_k\subset B(u_k,R)$ for $$R=3A_1\frac{\epsilon'}{|b|}.$$ Here we have used property $2$ from Proposition \ref{prop:Interval covers}. By property $4$ of Proposition \ref{prop:Interval covers}, we also have $B(v_k,R')\subset V_{j(k)}$ where 
			$$R'=\frac{1}{2}A_2A_1'\frac{\epsilon'}{|b|}$$ 
			and $v_k\in K_{w}\cap V_{j(k)}$ is such that 
			$$\dist(v_k,\partial V_{j(k)})=\dist(K_{w}\cap V_{j(k)},\partial V_{j(k)}).$$
			
			Let $H\in C_{A|b|}$. We have 
			$$\int_{K_{w}}H\, d\mu_{w}=\sum_{k\in\mathcal{G}}\int_{V_k}H\, d\mu_{w}\leq \sum_{k\in \mathcal{G}}\int_{B(u_k,R)}H\, d\mu_{w}\leq \sum_{k\in \mathcal{G}}(\max_{B(u_k,R)}H)\mu_{w}(B(u_k,R)).$$
			For our choice of $R$ and $R'$ we have 
			$$B(v_k,R')\subset B(u_k,R)\subset B(v_k,2R).$$ 
			Therefore, using Lemma \ref{lma:doubling}, it follows that there exists $C'>0$ depending only upon $A_1', A_1$ and $A_2$ such that $$\mu_{w}(B(u_k,R))\leq \mu_{w}(B(v_k,2R))\leq C'\mu_{w}(B(v_k,R'))\leq C'\mu_{w}(V_{j(k)}).$$ Now using the fact $$e^{-A|b||x-y|}\leq\frac{H(x)}{H(y)}\leq e^{A|b||x-y|}\footnote{This is a general property of functions in $C_{A|b|}$.}$$ for all $x,y\in I$, for $|b|$ large enough, we deduce 
			\begin{align*}
				\int_{K_w}H\, d\nu&\leq C'\sum_{k\in\mathcal{G}}e^{2A|b|R}(\min_{V_{j(k)}}H)\mu_{w}(V_{j(k)})\\
				&\leq C'e^{6AA_1\epsilon'}\sum_{k\in\mathcal{G}}\int_{V_{j(k)}}H\, d\mu_{w}\\
				&\leq 5C'e^{6AA_1\epsilon'}\sum_{j:\exists i\, s.t. (i,j)\in J}\int_{V_j}H\, d\mu_{w}\\
				&=5C'e^{6AA_1\epsilon'}\int_{W_J}H\, d\mu_{w}.
			\end{align*}
			Taking $\tilde{\epsilon}=5C'e^{6AA_1\epsilon'}$ completes our proof.
			
			
		\end{proof}
		
		We define $\EE_{b}$ to be those subsets $J\subset\mathcal{J}_{b}$ such that $J$ is dense. 
		
		\begin{prop}
			 There exists $0<\rho<1$ such that for all $w\in \cup_{n=1}^{\infty}\Omega^n$, $|r|$ sufficiently small and  $|b|$ sufficiently large, for all $H\in C_{A|b|}$ and for all $J\in \EE_{s}$ we have $$\int_{K_{w}}|\NN_{s}^{J}(H)|^2\, d\mu_{w} \leq \rho\int_{K_{(w^*)^Nw}} H^2\, d\mu_{(w^*)^Nw}.$$
		\end{prop}
		
		\begin{proof}
			Let $H\in C_{A|b|}$ and $J\in \EE_{s}$. For all $x\in I$, we have by the Cauchy-Schwartz inequality, \begin{align*}
				(\NN_{s}^{J}(H))^{2}(x)&=\left(\sum_{\a\in (w^*)^N}\frac{1}{2^N}|\phi_{\a}'(x)|^r \chi_{J}(\phi_{\a}(x))H(\phi_{\a}(x))\right)^2\\
				&\leq \left(\sum_{\a\in (w^*)^N}\frac{1}{2^N} |\phi_{\a}'(x)|^{2r} \chi_{J}^2(\phi_{\a}(x))\right)\left(\sum_{\a\in (w^*)^N}\frac{1}{2^N}  H^2(\phi_{\a}(x))\right)\\
				&=\left(\sum_{\a\in (w^*)^N}\frac{1}{2^N} |\phi_{\a}'(x)|^{2r} \chi_{J}(\phi_{\a}(x))\right)\cL_{w^*}^{N}(H^2)(x).
			\end{align*}
			For all $x\in W_J,$ for a well chosen $i$ we have $\chi_{J}(\varphi_{\alpha_{i}^{N}}(x))=1-\theta$. Therefore for $x\in W_{J}$ we have
			\begin{align*}
				\sum_{\a\in (w^*)^N}\frac{1}{2^N} |\phi_{\a}'(x)|^{2r} \chi_{J}(\phi_{\a}(x))&\leq \sum_{\a\neq \alpha_i^{N}}\frac{1}{2^N}  |\phi_{\a}'(x)|^{2r}\chi_{J}(\varphi_{\a}(x))+\frac{(1-\theta)}{2^N}  |\phi_{\alpha_i}'(x)|^{2r}\\
				& \leq   \sum_{a\neq \alpha_{i}^{N}}\frac{1}{2^N}\gamma_{1}^{|2r| N}+\frac{(1-\theta)}{2^N} \gamma_1^{|2r| N} \\
				&\leq \gamma_{1}^{|2r| N} \left(1-\frac{\theta}{2^N}\right).
			\end{align*}
			In the penultimate inequality we have used \eqref{eq:unifexp}. It can similarly be shown that for $x\notin W_{J}$ we have 
			$$\sum_{\a\in (w^*)^N}\frac{1}{2^N} |\phi_{\a}'(x)|^{2r} \chi_{J}(\phi_{\a}(x))\leq \gamma_{1}^{|2r|N}.$$
			Now 
			$$\int_{K_{w}}(\NN_{s}^{J}(H))^2\, d\mu_{w}=\int_{W_{J}}(\NN_{s}^{J}(H))^2\, d\mu_{w}+\int_{K_{w}\setminus W_J}(\NN_{s}^{J}(H))^2\, d\mu_{w}.$$
			Applying the inequalities above yields
			\begin{align*}
				\int_{K_{w}}(\NN_{s}^{J}(H))^2\, d\mu_{w}&\leq \gamma_{1}^{|2r| N} \left(1-\frac{\theta}{2^N}\right)\int_{W_J}\cL_{w}^{N}(H^2)\, d\mu_{w} + \gamma_{1}^{|2r| N}\int_{K_{w}\setminus W_J}\cL_{w}^{N}(H^2)\, d\mu_{w}\\
				&=\gamma_{1}^{|2r|N} \int_{K_{w}}\cL_{w^*}^{N}(H^2)\, d\mu_{w} - \frac{\gamma_{1}^{|2r|N}\theta}{2^N}\int_{W_J} \cL_{w^*}^{N}(H^2)\, d\mu_{w}.
			\end{align*} Applying Lemma \ref{W_J lemma} to $\cL_{w^*}^{N}(H^2)$, which is possible as $\cL_{w^*}^{N}(H^2)\in C_{3A|b|/4}$ when $N$ is sufficiently large, we have 
			$$\int_{K_{w}}(\NN_{s}^{J}(H))^2\, d\mu_w \leq \gamma_{1}^{|2r| N} \left(1-\frac{\tilde{\epsilon}\theta}{2^N}\right)\int_{K_{w}}\cL_{w^*}^{N}(H^2)\, d\mu_{w}.$$ 
			Our proof now follows by taking $\rho\in(0,1)$ so that $ \gamma_{1}^{|2r| N} \left(1-\frac{\tilde{\epsilon}\theta}{2^N}\right)\leq\rho$ for all $|r|$ sufficiently small, and using that $$\int_{K_{w}}\cL_{w^*}^{N}(H^2)\, d\mu_{w}=\int_{K_{(w^*)^Nw}}H^2\, d\mu_{(w^*)^Nw}.$$

		\end{proof}
		This proposition establishes property $2$ of Lemma \ref{lma:dolgopyat}. 
		
		Next, we now draw our attention to the first part of property 3 of Lemma \ref{lma:dolgopyat}. Here is where the nonlinearity of the IFS will manifest itself. We recall here the nonlinearity property of the IFS $w^*$ that follows from the discussion following Proposition \ref{prop:IFS partition}:	There exists $c_1,c_2,\delta>0$ such that for all $x\in \{x:d(x,K)<\delta\}$ and $l\in\mathbb{N}$ we have 
		\begin{equation}
			\label{prop:nonlinear}
			c_1\leq \left|\frac{\varphi_{\alpha_{1}^{l}}''(x)}{\varphi_{\alpha_{1}^{l}}'(x)}-\frac{\varphi_{\alpha_{2}^{l}}''(x)}{\varphi_{\alpha_{2}^{l}}'(x)}\right|\leq c_2.
		\end{equation}
	 Equipped with the nonlinearity UNI condition \eqref{prop:nonlinear}, we can now prove:
		

		\begin{lemma}
			\label{lma:Theta lemma}
			Let $s=r+ib$, $H\in C_{A|b|}$, $f\in C^{1}(I)$ be such that $|f|\leq H$ and $|f'|\leq A|b|H$. Define the functions $\Theta_i:I\to [0,\infty)$ for $i=1,2$, by 
			$$\Theta_{1}(x):=\frac{\left||\varphi_{\alpha_{1}^{N}}'(x)|^s f(\varphi_{\alpha_{1}^{N}}(x))+|\varphi_{\alpha_{2}^{N}}'(x)|^sf(\varphi_{\alpha_{2}^{N}}(x))|\right|}{(1-2\theta) |\varphi_{\alpha_{1}^{N}}'(x)|^r H(\varphi_{\alpha_{1}^{N}}(x))+|\varphi_{\alpha_{2}^{N}}'(x)|^r H(\varphi_{\alpha_{2}^{N}}(x))}$$
			and
			$$\Theta_{2}(x):=\frac{\left||\varphi_{\alpha_{1}^{N}}'(x)|^sf(\varphi_{\alpha_{1}^{N}}(x))+|\varphi_{\alpha_{2}^{N}}(x)|^sf(\varphi_{\alpha_{2}^{N}}(x))|\right|}{|\varphi_{\alpha_{1}^{N}}'(x)|^r H(\varphi_{\alpha_{1}^{N}}(x))+(1-2\theta)|\varphi_{\alpha_{2}^{N}}'(x)|^r H(\varphi_{\alpha_{2}^{N}}(x))}$$
			Then for $\theta,$ $|r|$ and $\epsilon'$ sufficiently small, for any $w\in \Omega^{*}$ for all $j$ such that $V_{j}\cap K_{w}\neq \emptyset$, there exists $j'$ with $|j'-j|\leq 2$, $V_{j'}\cap K_{w}\neq \emptyset$ and $i\in\{1,2\}$ such that for all $x\in V_{j'}$, we have $$\Theta_i(x)\leq 1.$$
		\end{lemma}
		This important lemma relies upon the following lemmas that are taken directly from Naud's paper \cite[Lemma 5.11 and Lemma 5.12]{Naud}:
		
		\begin{lemma}
			\label{Uniform bound}
			Let $Z\subset I$ be an interval with $|Z|\leq \frac{c}{|b|}$. Let $H\in C_{A|b|}$ and $f\in C^{1}(I)$ satisfy $|f|\leq H$ and $|f'|\leq A|b|H$. Then for $c$ sufficiently small, we have either $|f(u)|\leq \frac{3}{4}H(u)$ for all $u\in Z$, or $|f(u)|\geq\frac{1}{4} H(u)$ for all $u\in Z$.
		\end{lemma}
		\begin{lemma}
			\label{Triangle lemma} 
			Let $z_1,z_2\neq 0$ be two complex numbers such that $\left|\frac{z_1}{z_2}\right|\leq L$ and $2\pi -\epsilon \geq |arg(z_1)-arg(z_2)|\geq \epsilon>0$. Then there exists $0<\delta(L,\epsilon)<1$ such that $$|z_1+z_2|\leq (1-\delta)|z_1|+|z_2|.$$
		\end{lemma}
		
		\begin{proof}[Proof of Lemma \ref{lma:Theta lemma}]
			Let $\epsilon'$ be sufficiently small so that Lemma \ref{Uniform bound} holds for all $Z=Z_{j}^{i}$. We assume that $0<\theta<1/8$. We have $|Z_{j}^{i}|\leq |V_j|\cdot \gamma^{-N}$ so we can always assume $|Z_{j}^i|\leq |V_j|.$
			
			Let $V_j, V_{j+1}, V_{j+2}$ be a triple of intervals each with non-empty intersection with $K_{w}$. Let $\hat{V}_j=V_j\cup V_{j+1}\cup V_{j+2}$. We assume that $\epsilon'$ is sufficiently small so that $\hat{V}_{j}\subset\{x:d(x,K)<\delta\},$ and therefore \eqref{prop:nonlinear} applies to elements in $\hat{V}_j$.

			Two cases occur. If there exists $j'\in \{j,j+1,j+2\}$ such that $|f(u)|\leq \frac{3}{4}H(u)$ for all $u\in Z_{j'}^{i}$ for some $i\in \{1,2\}$, then $\Theta_i(x)\leq 1$ for all $x\in V_{j'}$ (here we are using that $\theta<1/8$). If this is not the case, then by Lemma \ref{Uniform bound} we have for all $j'\in \{j,j+1,j+2\},$ for all $i\in\{1,2\}$ and for all $u\in Z_{j'}^{i}$, 
			\begin{equation}
				\label{lower bound}
				|f(u)|\geq \frac{1}{4}H(u).
			\end{equation}
			We now set out to apply Lemma \ref{Triangle lemma} to complete our proof.
			
			For $x\in \hat{V}_j$, we set 
			
			$$z_{1}(x)=|\varphi_{\alpha_{1}^{N}}'(x)|^sf(\varphi_{\alpha_{1}^{N}}(x)) \textrm{ and }z_{2}(x)=|\varphi_{\alpha_{2}^{N}}'(x)|^sf(\varphi_{\alpha_{2}^{N}}(x)).$$ 
			
			We claim that given $j'\in \{j,j+1,j+2\}$, we have either $\left|\frac{z_1(x)}{z_2(x)}\right|\leq M$ for all $x\in V_{j'}$ or $\left|\frac{z_2(x)}{z_1(x)}\right|\leq M$ for all $x\in V_{j'}$ for some $M>0$. Using \eqref{lower bound}, our assumptions on $f$ and \eqref{eq:unifexp}, we see that for all $x\in V_{j'}$ we have 
			
			$$\gamma^{-|2r|N} \frac{H(\varphi_{\alpha_1^{N}}(x))}{4H(\varphi_{\alpha_2^{N}}(x))}\leq \left|\frac{z_1(x)}{z_2(x)}\right|\leq \gamma^{|2r|N}  \frac{4H(\varphi_{\alpha_1^{N}}(x))}{H(\varphi_{\alpha_2^{N}}(x))}.$$ Taking $|r|$ sufficiently small so that $|r|<1/2,$ we see that the above immediately implies
			$$\gamma^{-N} \frac{H(\varphi_{\alpha_1^{N}}(x))}{4H(\varphi_{\alpha_2^{N}}(x))}\leq \left|\frac{z_1(x)}{z_2(x)}\right|\leq \gamma^{N}  \frac{4H(\varphi_{\alpha_1^{N}}(x))}{H(\varphi_{\alpha_2^{N}}(x))}.$$ 
			
			If there exists $x_0\in V_{j'}$ such that 
			\begin{equation}
				\label{lessone}
				\frac{H(\varphi_{\alpha_1^N}(x_0)))}{H(\varphi_{\alpha_2^N}(x_0))}\leq 1,
			\end{equation}
			then for all $x\in V_{j'}$ we have 
			$$\frac{H(\varphi_{\alpha_1^N}(x))}{H(\varphi_{\alpha_2^N}(x))}\leq \frac{e^{AA_1\epsilon'}H(\varphi_{\alpha_1^N}(x_0)))}{e^{-AA_1\epsilon'}H(\varphi_{\alpha_2^N}(x_0))}\leq e^{2AA_1\epsilon'}.$$ Here we are using property 2 from Proposition \ref{prop:Interval covers}, the fact $|Z_{j}^i|\leq |V_j|,$ and the inequality $$e^{-A|b||x-y|}\leq\frac{H(x)}{H(y)}\leq e^{A|b||x-y|}$$ for all $x,y\in I$. Therefore if \eqref{lessone} holds for some $x_0\in V_{j'}$ then 
			$$\left|\frac{z_1(x)}{z_2(x)}\right|\leq 4\gamma^{N} e^{2AA_1\epsilon'}=:M.$$ 
			If $$\frac{H(\varphi_{\alpha_1^N}(x)))}{H(\varphi_{\alpha_2^N}(x))}\geq 1$$ for all $x\in V_{j'},$ then it can similarly be shown that 
			$$\left|\frac{z_2(x)}{z_1(x)}\right|\leq 4\gamma^{N} e^{2AA_1\epsilon'}.$$ This completes our proof of the claim.

			We now try to control the variations of the arguments of $z_1$ and $z_2$. Since $|z_i(x)|\geq \frac{\gamma^{-|r|N}}{4}H(\varphi_{\alpha_{i}^{N}}(x))>0$ for all $x\in \hat{V}_j$ and $i=1,2$, there exist two $C^{1}$ functions $L_i:\hat{V}_j\to \mathbb{C}$ such that for $i=1,2$ we have $L_{i}'(x)=\frac{z_i'(x)}{z_i(x)}$ and $e^{L_i(x)}=z_i(x)$ for all $x\in \hat{V_j}$\footnote{Details on how to construct the $L_i$ are given in \cite{Naud}.}.
			
			Let 
			$$\Phi(x)=Im(L_1(x))-Im(L_2(x)).$$
			Taking derivatives, for all $x\in \hat{V}_j$ we get
			\begin{align*}
				\Phi'(x)&=Im\left(\frac{z_1'(x)}{z_1(x)}-\frac{z_2'(x)}{z_2(x)}\right)\\
				&=b\left(\frac{\varphi_{\alpha_1^N}''(x)}{\varphi_{\alpha_1^N}'(x)}-\frac{\varphi_{\alpha_2^N}''(x)}{\varphi_{\alpha_2^N}'(x)}\right) +Im\left(\frac{(f\circ \varphi_{\alpha_1^N})'(x)}{f(\varphi_{\alpha_1^N}(x))}-\frac{(f\circ \varphi_{\alpha_2^N})'(x)}{f(\varphi_{\alpha_2^N}(x))}\right).
			\end{align*}
			Using \eqref{lower bound} and our assumptions on $f$ we have 
			$$\left|\frac{(f\circ \varphi_{\alpha_1^N})'(x)}{f(\varphi_{\alpha_1^N}(x))}-\frac{(f\circ \varphi_{\alpha_2^N})'(x)}{f(\varphi_{\alpha_2^N}(x))}\right|\leq 8A|b|\gamma^{-N}.$$ Recall that $\hat{V}_{j}\subset \{x:d(x,K)<\delta\}$ where $\delta$ is as in the statement of \eqref{prop:nonlinear}. 
			Hence by the UNI condition \eqref{prop:nonlinear}, for all $x\in \hat{V}_j$ we have $$c_1-8A\gamma^{-N}\leq \frac{|\Phi'(x)|}{|b|}\leq c_2 +8A\gamma^{-N}.$$ For $x\in V_{j}$ and $x'\in V_{j+2}$, we now have by Proposition \ref{prop:Interval covers} and the mean value theorem that
			$$\left(c_1-8A\gamma^{-N}\right)A_1'\epsilon'\leq|\Phi(x)-\Phi(x')|\leq \left(c_2 +8A\gamma^{-N}\right)3A_1\epsilon'.$$ By choosing $N$ large enough, we see that there exists $B_1,B_2$ independent of $x,x'$ and $|b|$ such that $$B_1\epsilon'\leq |\Phi(x)-\Phi(x')|\leq B_2\epsilon'.$$
			We now choose $\epsilon'$ such that $(B_2+B_1/2)\epsilon'\leq \pi$ and set $\epsilon=B_1\frac{\epsilon'}{4}$. Suppose now that there exists $x\in V_j$ and $x'\in V_{j+2}$ such that both $$\Phi(x),\Phi(x')\in \cup_{k\in\mathbb{Z}}[2k\pi-\epsilon,2k\pi+\epsilon].$$ Since $|\Phi(x)-\Phi(x)|\leq B_2\epsilon'$, we cannot have $$\Phi(x)\in [2k_1\pi-\epsilon, 2k_1\pi+\epsilon]\textrm{ and }\Phi(x')\in[2k_2\pi-\epsilon,2k_2\pi+\epsilon])$$ with $k_1\neq k_2$. As in that case we would have $$B_2\epsilon'\geq |\Phi(x)-\Phi(x')|\geq 2\pi-2\epsilon=2\pi-B_1\epsilon'/2,$$ which is not possible by our choice of $\epsilon'$. Therefore we have $$B_1\epsilon'\leq |\Phi(x)-\Phi(x')|\leq 2\epsilon=B_1\epsilon'/2$$ which is again a contradiction. Therefore there exists $j'\in\{j,j+2\}$ such that for all $x\in V_{j'}$, $d(\Phi(x),2\pi\mathbb{Z}))>\epsilon.$ Because $e^{i(\Phi(x))}=e^{i(arg(z_1)-arg(z_2))}$, the hypothesis of Lemma \ref{Triangle lemma} are satisfied. We get either for all $x\in V_{j'}$ $$|z_1(x)+z_2(x)|\leq (1-\delta(M,\epsilon))|z_1(x)|+|z_2(x)|,$$
			or for all $x\in V_{j'}$ 
			$$|z_1(x)+z_2(x)|\leq (1-\delta(M,\epsilon))|z_2(x)|+|z_1(x)|,$$ depending on whether $\left|\frac{z_1(x)}{z_2(x)}\right|\leq M$ or $\left|\frac{z_2(x)}{z_1(x)}\right|\leq M$. By choosing $0<\theta<\frac{1}{2}\delta(M,\epsilon)$ we have $\Theta_i(x)\leq 1$ for some $i\in\{1,2\}$ for all $x\in V_{j'}$.

		\end{proof}
		Now we can prove the first part of property $3$ from Lemma \ref{lma:dolgopyat}. 
		
		\begin{proof}[Proof of the first part of the property 3 of Lemma \ref{lma:dolgopyat}] Fix $w\in \cup_{n=1}^{\infty}\Omega^n$. We assume that the constants have been chosen so that property $1$, property $2$, and Lemma \ref{lma:Theta lemma} are satisfied. Let $f\in C^{1}(I)$ and $H\in C_{A|b|}$ with $|f|\leq H$ and $|f'|\leq A|b|H$. We must show that there exists a dense subset $J\in \EE$ such that 
			
			$$|\cL^{N}_{s,w^*}(f)|\leq \NN_s^{J}(H).$$ 
			
			Let $J$ be the set of indexes $(i,j)$ such that $\Theta_i(x)\leq 1$ for all $x\in V_j$. Lemma \ref{lma:Theta lemma} tells us that $J$ is dense. Let $x\in I$. If $x\notin \mathrm{int}\,V_j$ for any $j$ such that $(i,j)\in J$ for some $i\in\{1,2\}$, then $\chi_{J}(\phi_{\a}(x))=1$ for all $\a\in (w^*)^N$. This is because $\phi_{\a}(x)\in Z_{j}^{i}$ if and only if $\a=\alpha_{i}^N$ and $x\in V_j$ for some $(i,j)\in J$. Therefore for $x\notin \mathrm{int}\,V_j$ we have
			\begin{align*}
				|\cL^{N}_{s,w^*}(f)(x)| &\leq \sum_{\a\in (w^*)^N}\frac{1}{2^N}|\phi_{\a}'(x)|^r H(\phi_{\a}(x))\\
				&=\sum_{\a\in (w^*)^N}\frac{1}{2^N} |\phi_{\a}'(x)|^r \chi_{J}(\phi_{\a}(x))H(\phi_{\a}(x))\\
				&=\NN_{s}^{J}(H)(x).
			\end{align*}
			
			If $x\in \mathrm{int}\,V_j$ for some $j$ for which there exists $i\in\{1,2\}$ such that $(i,j)\in J,$ then we apply the following argument. 
			\begin{enumerate}
				\item If $(1,j)\in J$ and $(2,j)\notin J$, then $\chi_{J}(\phi_{\a}(x))=1$ for all $\a\in (w^*)^N$ such that $a\neq \alpha_{1}^{N}$. Now using the fact that $\Theta_1(x)\leq 1$, we get
				\begin{align*}
					|\cL^{N}_{s,w^*}(f)(x)|&\leq \sum_{\stackrel{\a\in (w^*)^N}{\a\neq \alpha_{1}^N,\alpha_{2}^N}}\frac{1}{2^N} |\phi_{\a}'(x)|^r H(\phi_{\a}(x))\\
					& \qquad +\frac{(1-2\theta)}{2^N}|\phi_{\alpha_{1}^N}'(x)|^r H(\varphi_{\alpha_{1}^N}(x))+\frac{1}{2^N}|\phi_{\alpha_{2}^N}'(x)|^r H(\varphi_{\alpha_2^{N}}(x))\\
					&\leq \sum_{\a\in (w^*)^N}\frac{1}{2^N}|\phi_{\a}(x)|^r \chi_{J}(\phi_{\a}(x))H(\phi_{\a}(x))\\
					& =\NN_{s}^{J}(H)(x).
				\end{align*}
				The case $(2,j)\in J$ and $(1,j)\notin J$ is symmetric. 
				\item If $(1,j)\in J$ and $(2,j)\in J$, then $\chi_{J}(\phi_{\a}(x))=1$ for all $\a\notin \{\alpha_{1}^N,\alpha_{2}^N\}$. In addition $\Theta_1(x)\leq 1$ and $\Theta_{2}(x)\leq 1$. Combining these two inequalities we deduce
				\begin{align*}
					&\left||\phi_{\alpha_{1}^N}'(x)|^s f(\varphi_{\alpha_{1}^N}(x))+|\phi_{\alpha_{2}^N}'(x)|^s f(\varphi_{\alpha_2^N}(x))\right|\\
					\leq&(1-\theta)|\phi_{\alpha_{1}^N}'(x)|^rH(\varphi_{\alpha_{1}^N}(x))+(1-\theta)|\phi_{\alpha_{2}^N}'(x)|^rH(\varphi_{\alpha_{2}^N}(x))\\
					\leq&   |\phi_{\alpha_{1}^N}'(x)|^r \chi_{J}(\varphi_{\alpha_{1}^N}(x))H(\varphi_{\alpha_{1}^N}(x))+ |\phi_{\alpha_{2}^N}'(x)|^r \chi_{J}(\varphi_{\alpha_{2}^N}(x))H(\varphi_{\alpha_{2}^N}(x)).
				\end{align*}
				This implies that 
				$$|\cL^N_{s,w^*}(f)(x)|\leq \NN_{s}^{J}(H)(x).$$
			\end{enumerate} This complete our proof of the first part of property $3$ from Lemma \ref{lma:dolgopyat}.
		\end{proof}
		
		Now we will focus on the second part of property $3$ from Lemma \ref{lma:dolgopyat}.
		
		\begin{proof}[Proof of the second part of property 3 of Lemma \ref{lma:dolgopyat}] 
			Let $f\in C^{1}(I)$ and $H\in C_{A|b|}$ be such that $|f|\leq H$ and $|f'|\leq A|b|H.$ Assume now $|b| > 1$ and $|r| < 1$. Then we have:
			\begin{align*}
				&|\cL^{N}_{s,w^*}(f)'(x)|\\
				&\leq \sum_{\a\in (w^*)^N}\frac{1}{2^N} |\phi_{\a}'(x)|^r \left(|(f\circ \varphi_{\a})'(x)|+\frac{|\phi_{\a}''(x)|}{|\phi_{\a}'(x)|}|s| |f(\varphi_{\a}(x))|\right)\\
				& \leq \sum \frac{1}{2^N} |\phi_{\a}'(x)|^r\left(|f'(\varphi_{\a}(x))|\gamma^{-N}+C_0 |s|  |H(\varphi_{\a}(x))|\right)\\
				&\leq \sum \frac{1}{2^N} |\phi_{\a}'(x)|^r\left(A|b|H(\varphi_{\a}(x))\gamma^{-N}+2C_0 |b|H(\varphi_{\a}(x))|\right)\\
				&=A|b|\gamma^{-N}\cL^{N}_{r,w^*}(H)(x)+2C_0|b|\cL^{N}_{r,w^*}(H)(x)\\
				&\leq \frac{A|b|\gamma^{-N}}{1-\theta}\NN_{s}^{J}(H)(x)+\frac{2C_0|b|}{1-\theta}\NN_{s}^{J}(H)(x)
			\end{align*} In the last line we used that $H\leq \frac{1}{1-\theta}\chi_{J} H$ giving 
			$$\cL^{N}_{r,w^*}(H)(x)\leq \frac{1}{1-\theta}\NN_{s}^{J}(H)(x).$$
			It follows now that for $A>8C_0,$ $\theta<1/2$, and $N$ sufficiently large so that $\gamma^{-N}<1/4$,  we have 
			$$|\cL^{N}_{s,w^*}(f)'(x)|\leq A|b|\NN_{s}^{J}(H)(x).$$ 
			This establishes the second part of property $3$ from Lemma \ref{lma:dolgopyat}. 
		\end{proof}
		
		This completes the proof of Lemma \ref{lma:dolgopyat} and thus the proof of the spectral gap Theorem \ref{thm:Linftybound} is complete.

 \section{Proof of the Fourier decay theorem}\label{sec:proofmain}

Assuming the spectral gap Theorem \ref{thm:Linftybound} holds, let us now show how to prove Theorem \ref{thm:main1}. Our main task is to reduce the quantity $|\widehat{\mu}(\xi)|^2$ using Cauchy-Schwartz, the mean value theorem and certain large deviation bounds into an exponential sum.  We can then apply the following general exponential sum bound for non-concentrated products of real numbers. This bound is a corollary of the sum-product theorem \cite{Bourgain2010}. This specific form is taken from \cite[Lemma 4.3]{SS}:

\begin{thm}[Bound for exponential sums of non-concentrated products]
\label{lma:discretised}
Fix $\eps_0 > 0$. Then there exist $k \in \N$ and $\eps_1>0,\eps_2 > 0$ depending only on $\eps_0$ such that the following holds. 

Fix $\eta \in \R$ such that $|\eta| > 1$.  Let $R,N > 1$ and $\cZ_1,\dots,\cZ_k$ be finite sets such that $\sharp \cZ_j \leq R N$.  Suppose $\zeta_j$, $j = 1,\dots,k$, are real valued functions on the sets $\cZ_j$ that satisfy for all $j = 1,\dots,k$ that
\begin{itemize}
\item[(1)] the range
$$\zeta_j(\cZ_j) \subset [R^{-1},R];$$
\item[(2)] for all $\sigma \in [R^{-2}|\eta|^{-1},|\eta|^{-\eps_1}]$
$$\sharp\{(\b,\c) \in \cZ_j^2 : |\zeta_j(\b) - \zeta_j(\c)| \leq \sigma\} \leq N^2 \sigma^{\eps_0}.$$
\end{itemize}
Then there exists a constant $c > 0$ depending only on $k$ such that
$$\Big|N^{-k} \sum_{\b_1 \in \cZ_1,\dots,\b_k \in \cZ_k} \exp(2\pi i  \eta \zeta_1(\b_1) \dots \zeta_k(\b_k))\Big| \leq c R^{k} |\eta|^{-\varepsilon_2}.$$
\end{thm}

In other words, if the numbers $\zeta_j(\b_j)$ do not concentrate too much in scales roughly between $|\eta|^{-1}$ and $|\eta|^{-\eps_1}$, then the corresponding exponential sums for the products 
$$\zeta_1(\b_1) \dots \zeta_k(\b_k)$$ 
at frequency $\eta$ have to decay with for some power of $|\eta|$. For us, the mappings $\zeta_j : \cZ_j \to [R^{-1},R]$ appear from a multiscale decomposition of $\mu$ when we iterate the self-conformality. In order to define them, we first need some notations and parameters.

\begin{notations} Let $\epsilon>0$ and $\mathcal{R}_{n}^k(\epsilon,\eps_0)$ for $\eps > 0$, $n \in \N$, $\eps_0 > 0$ and $k \in \N$ be defined by the set of blocks (concatenations of words in $\A^n$) $\a_1 \dots \a_k \in (\A^{n})^k$ where
$$\a_j \in \cR_n(\eps,\eps_0) := \bigcap\limits_{l=\lfloor \eps_0 n \rfloor}^n \{ {\a} \in \A^n : [a_1\ldots a_l]\subset A_l(\varepsilon) \}, \quad j = 1,\dots,k,$$
where, using
$$\tau(\a) := -\log |\phi_{a_1}' (\pi ( (a_{i+1})))|$$
and
$$\psi(\a) := -\log p_{a_1}$$
 we define:
$$A_n(\eps) := \Big\{\a\in \A^{\N} : \Big| \frac{1}{n} S_n \tau (\a) - \lambda \Big| < \varepsilon \quad  \mathrm{ and } \quad  \Big| \frac{1}{n}S_{n}\psi(\a) - h \Big| < \varepsilon \Big\}.$$ 
Here the Lyapunov exponent and entropy of $\mu$ are given by
$$\lambda:=\int \tau(\a)\, dm_{\p}\qquad \textrm{ and }\qquad h:=-\sum_{a\in \A}p_a\log p_a.$$ 
\end{notations}

Applying the large deviation principle (see e.g. \cite[Theorem 4.1]{JordanSahlsten}) and the arguments given in the proof of \cite[Lemma 2.2]{SS}, we see that the elements of $\cR_n(\eps,\eps_0)$ and $\cR_{n}^{k}(\eps,\eps_0)$ satisfy the following useful properties.

\begin{lemma}[Regularity and large deviations]	\label{e:large deviation} For any $\eps,\eps_0 > 0$ and $k,n \in \N$, if $\a\in \cR_n(\eps,\eps_0)$, then for any $\lfloor \epsilon_0 n\lfloor \leq j\leq n$ we have
\begin{equation}
	\label{e:derweights} e^{-\eps j} e^{-\lambda j} \lesssim |\phi_{\a|_j}'(x)| \lesssim e^{\eps j} e^{-\lambda j}\quad \text{for all }x\in I \quad \text{and} \quad e^{-\eps j} e^{-h j} \lesssim p_{\a|_j}\lesssim e^{\eps j} e^{-h j}.
	\end{equation}
Furthermore,
\begin{equation}
	\label{e:cards} 
e^{\eps n} e^{h n} \lesssim  \sharp \mathcal{R}_{n}(\epsilon,\eps_0) \lesssim e^{\eps n} e^{h n} \quad \text{and} \quad e^{\eps k n} e^{hk n} \lesssim  \sharp \mathcal{R}_{n}^k(\epsilon,\eps_0)\lesssim e^{\eps k n} e^{hk n}
\end{equation}
and
\begin{equation}
	m_{\p}( \mathcal{R}_{n}^k(\epsilon,\eps_0))\lesssim ke^{-\delta n}
\end{equation}for some $\delta>0$.
\end{lemma}
Note that the proof in \cite{JordanSahlsten} is formulated in terms of the measure $\mu_{\p}$ rather than $m_{\p}.$ However, because the underlying IFS they are considering satisfies the strong separation condition, the proof translates over into our symbolic setting with exactly the same proof.

Given now the notation on regular words, we will use the following parameters, which we can use to define the maps $\zeta_j$ needed for Theorem \ref{lma:discretised}. This also helps to keep track of all the various constants, coefficients and their dependencies that might have been hard to track in \cite{SS}:

\begin{parameters}[All the parameters and their dependencies]\label{rmk:data}
\begin{itemize}
\item[(i)] Given the IFS $\Phi = \{\phi_a : a \in \A\}$ and self-conformal measure $\mu$, let $\lambda > 0$ be the Lyapunov exponent of $\p$, $h > 0$ be the entropy of $\p$ and $s = h/\lambda > 0$. Recall that $ \gamma > 1$ is the maximal uniform contraction rate of the maps $\phi_a$ and $B > 0$ is the bounded distortion constant.
\item[(ii)] This IFS now fixes the family of perturbed transfer operators $\cL_{s}$, $s \in \mathbb{C}$. For this family, let $\rho_0 > 0$ be the uniform spectral gap from Theorem \ref{thm:Linftybound} that exists for $|r|$ sufficiently small and $|b|$ sufficiently large (note that crucially it does not depend on $b$) and set
$$\eps_0 := \frac{1}{2}\min\{\log (1/\rho_0),\lambda/2\} > 0,$$ 
which fixes once and for all by Theorem \ref{lma:discretised} the parameters $k \in \N$, $\eps_1 > 0$ and $\eps_2>0$.
\item[(iii)] Now, using the data $\eps_0,k,\lambda$, let us introduce a way for us to fix the length of the words $n$ we will use when considering a given frequency $\xi \in \R$, $\xi \neq 0$. We define
$$n := \lfloor ((2k+1)\lambda + \epsilon_{0}) \log |\xi| \rfloor.$$ 
Here we assume (depending on $\lambda$) that $\xi$ is large enough such that $n > 1$. Note that now $|\xi|\sim e^{((2k+1)\lambda + \epsilon_{0})n }$. 
\item[(iv)] Given all of the above data, we will end up getting multiplicative error terms of the form $\exp(\beta_j \eps n)$ where coefficients $\beta_{j} > 0$ depend only on all the above data and artefacts of the estimates such as Cauchy-Schwartz inequality. In the end of the proof, we gather all these multiplicative errors into a single one, $\exp(\beta_0 \eps n)$, and we end up with an estimate
$$|\widehat{\mu}(\xi)| \lesssim e^{\beta_0 \eps n} e^{-\alpha_0 n}$$
for some $\alpha > 0$. Thus to get polynomial Fourier decay, one simply has to pick any $0 < \eps < \frac{\alpha_0}{\beta_0}$.
\item[(v)] We can now define what parameters we use in Theorem \ref{lma:discretised}. Assuming we have fixed $\eps > 0$, then define for $n \in \N$ a collection
$$ J_n(\eps,\eps_0) := \{ \eta \in \mathbb{R}:e^{\frac{\eps_0}{2} n} \leq |\eta| \leq  e^{(\eps_0+\eps) n}  \}$$
which will constitute the range of $\eta$ to which we will apply Theorem \ref{lma:discretised} with the choice of inputs:
\begin{enumerate}
	\item $R := R_{\eps,n} =  e^{\eps n}$.
	\item $\cZ_j := \cR_n(\eps,\eps_0)$ for all $j = 1,\dots,k$.
	\item The maps $\zeta_j := \zeta_{j,\a} : \cR_n(\eps,\eps_0) \to [R^{-1},R]$ will be defined by
	$$\zeta_{j,\a}(\b) := e^{2\lambda n} |\varphi_{\a_{j-1} \b}'(x_{\a_j})|, \quad \b \in \cR_n(\eps,\eps_0),$$
	where $\a = \a_0 \a_1\dots \a_k \in \A^{n(k+1)}$ and $x_{\a_j}$ is the centre point of the interval $\varphi_{\a_j}(I)$. They do indeed map $\cR_{n}(\eps,\eps_0)$ to $[R^{-1},R]$ by \eqref{e:derweights}.
\end{enumerate}
\end{itemize}
\end{parameters}

The main consequence of the spectral gap Theorem \ref{thm:Linftybound} is the following multiscale non-concentration estimate that we then feed into Theorem \ref{lma:discretised} to eventually establish Theorem \ref{thm:main1}.

\begin{prop}[Multiscale non-concentration]\label{lma:distribution}
Let $\cW$ be the set of $(k+1)$-tuples $\a \in \mathcal{R}_n^{k+1}(\eps,\eps_0)$ such that for all $j=1,\dots,k$, $\eta \in J_n(\eps,\eps_0)$ and $\sigma \in [R^{-2}|\eta|^{-1},|\eta|^{-\eps_1}]$, we have that
	$$ \sharp \{ ({\bf b, c}) \in \mathcal{R}_n(\eps,\eps_0)^2: |\zeta_{j,{\a}}({\b})-\zeta_{j,{\a}}({\c})| \leq \sigma \}  \leq \sharp \cR_n(\eps,\eps_0)^2 \sigma^{\eps_0/4}. $$
Then there exists $\kappa_0>0$ such that
	$$ \frac{\sharp(\mathcal{R}_n^{k+1}(\eps,\eps_0)\setminus \mathcal{W})}{\sharp \mathcal{R}_n^{k+1}(\eps,\eps_0)} \lesssim  k  e^{\epsilon\kappa_{0}n-\epsilon_{0}^2\epsilon_{1}n/24}.$$
\end{prop}

\begin{proof}
	We will split our proof into four steps.

\bigskip	
\noindent \textbf{Step 1. Deriving our proposition from \eqref{eq:triple}.} 
The main estimate we will need for the proof of this proposition is the following:
 
\medskip 
\textit{There exists $\kappa_0 > 0$ such that for all $\eps > 0$, $n \in \N$, $\eta \in J_n(\eps,\eps_0)$, $\sigma \in [R^{-2}|\eta|^{-1},|\eta|^{-\eps_1}]$, $x \in I$ we have
\begin{align}\label{eq:triple}\sharp\{(\a,\b,\c) \in \cR_n(\eps,\eps_0)^3 : |e^{2\lambda n}|\varphi_{{\bf ab}}'(x)|-e^{2\lambda n}|\varphi_{{\bf ac}}'(x)|| \leq \sigma \} \lesssim e^{\kappa_0 \eps n} \sigma^{\eps_0/3} \sharp \cR_n(\eps,\eps_0)^3.\end{align}}

\medskip
Indeed assuming \eqref{eq:triple} now holds, we can now conclude our proposition as follows. If $\eta \in J_n(\eps,\eps_0)$ and $\sigma \in  [R^{-2}|\eta|^{-1},|\eta|^{-\eps_1}]$, there is a unique $l\geq \lfloor \frac{\epsilon_{0}\epsilon_1 n}{2\log 2}\rfloor -1$ such that 
$$2^{-l-1}\leq \sigma \leq 2^{-l}.$$ 
For such an $l$, let $\mathcal{R}_l^*$ be the collection of pairs $(\a,\d) \in \mathcal{R}_n(\eps,\eps_0)^2$ such that
	$$ \sharp \cR_n(\eps,\eps_0)^{-2}\cdot  \sharp \{ (\b,\c) \in \cR_n(\eps,\eps_0)^2 : ||\varphi'_{\a\b}(x_{\d})|-|\varphi_{\a\c}'(x_{\d})||\leq e^{-2\lambda n} 2^{-l} \} \leq 2^{-(l+1)\eps_0/4}. $$
	 Using this terminology, if we have a block $\a$ such that $(\a_{j-1},\a_j)\in \mathcal{R}_l^*$ for every $j=1,...,k$ and every $l\geq \lfloor \frac{\epsilon_{0}\epsilon_1 n}{2\log 2}\rfloor -1$, then by the definition of $\mathcal{R}_l^*$ and by definition of $\zeta_{j,{\a}}(\b)$ we have that
	\begin{align*}
	& \sharp \cR_n(\eps,\eps_0)^{-2}\cdot\sharp \{ (\b,\c) \in \mathcal{R}_n(\eps,\eps_0)^2: |\zeta_{j,{\A}}({\b})-\zeta_{j,{\A}}({\c})| \leq \sigma \} \\
	&  \leq \sharp \cR_n(\eps,\eps_0)^{-2} \cdot\sharp \{ (\b,\c) \in \mathcal{R}_n(\eps,\eps_0)^2: ||\varphi_{\a_{j-1}\b}'(x_{\a_j})|-|\varphi_{\a_{j-1}\c}'(x_{\a_j})|| \leq e^{-2\lambda n} 2^{-l} \} \\
	& \leq 2^{-(l+1)\eps_0/4} \\
	& \leq \sigma^{\eps_0/4}
	\end{align*}
for all $\sigma \in [R^{-2}|\eta|^{-1},|\eta|^{-\epsilon_1}]$ and $\eta\in J_{n}(\epsilon,\epsilon_0)$. Thus we have the inclusion $$\left\{ \a\in \mathcal{R}_n^{k+1}(\eps,\eps_0) : (\a_{j-1},\a_j) \in \mathcal{R}_l^* \textrm{ for all }l\geq \left\lfloor \frac{\epsilon_{0}\epsilon_1 n}{2\log 2}\right\rfloor -1,j=1,\ldots k\right\} \subset \mathcal{W}.$$ So a block $\a \notin \mathcal{W}$ if there exists at least one $j$  and $l$ such that $(\a_{j-1},\a_j)\notin \mathcal{R}_l^*$. On the other hand, using Markov's inequality for $$f(\a,\d)= \sharp \cR_n(\eps,\eps_0)^{-2} \cdot \sharp \{ (\b,\c) \in \cR_n(\eps,\eps_0)^2 : ||\varphi'_{\a\b}(x_{\d})| - |\varphi'_{\a\c}(x_{\d})||\leq e^{-2\lambda n}2^{-l} \}$$ gives us 
	\begin{align*}
	 \sharp\{ \mathcal{R}_n(\eps,\eps_0)^2 \setminus \mathcal{R}_l^* \} & = \sharp\{ (\a,\d)\in \mathcal{R}_n(\eps,\eps_0)^2 : f(\a,\d) \geq 2^{-(l+1)\eps_0/4} \}\\&\leq 2^{(l+1)\eps_0/4}\sum_{(\a,\d) \in \mathcal{R}_n(\eps,\eps_0)^2} f(\a,\d)\\
	& \lesssim \frac{2^{(l+1)\eps_0/4}\sharp\{ (\a,\b,\c,\d) \in \cR_n(\eps,\eps_0)^4: ||\varphi'_{\a\b}(x_{\d})|-|\varphi'_{\a\c}(x_{\d})|| \leq e^{-2\lambda n} 2^{-l} \}}{\sharp \cR_n(\eps,\eps_0)^{2}}
	\end{align*}
Applying now \eqref{eq:triple} with $\sigma = 2^{-l}$ gives us
\begin{align*}&\frac{2^{(l+1)\eps_0/4}\sharp\{ (\a,\b,\c,\d) \in \cR_n(\eps,\eps_0)^4: ||\varphi'_{\a\b}(x_{\d})|-|\varphi'_{\a\c}(x_{\d})|| \leq e^{-2\lambda n} 2^{-l} \}}{\sharp \cR_n(\eps,\eps_0)^{2}} \\
& \lesssim e^{\eps \kappa_0 n} 2^{-l \eps_0/12}\sharp \cR_n(\eps,\eps_0)^{2}.
\end{align*}
This in turn implies
\begin{equation}
	\label{e:final count}
	\sharp( \mathcal{R}_n(\eps,\epsilon_0)^2\setminus \mathcal{R}^*_l ) \lesssim e^{\eps \kappa_0 n} 2^{-l \eps_0/12}\cdot\sharp \cR_n(\eps,\eps_0)^{2}.
\end{equation}
Applying \eqref{e:final count} we now have
\begin{align*}
	\sharp(\mathcal{R}_n^{k+1}(\eps,\eps_0)\setminus \mathcal{W})&\leq \sum_{j=1}^{k}\sum_{l=\lfloor \frac{\epsilon_{0}\epsilon_1 n}{2\log 2}\rfloor -1}^{\infty}\sharp\{\a\in\mathcal{R}_n^{k+1}(\eps,\eps_0): (\a_{j-1},\a_{j})\notin \cR_{l}^* \}\\
	&=\sum_{j=1}^{k}\sum_{l=\lfloor \frac{\epsilon_{0}\epsilon_1 n}{2\log 2}\rfloor -1}^{\infty}\sharp \mathcal{R}_n^{k-1}(\epsilon,\epsilon_0)\cdot \sharp( \mathcal{R}_n(\eps)^2\setminus \mathcal{R}^*_l )\\
	&\leq 	 \sharp \mathcal{R}_n^{k+1}(\epsilon,\epsilon_0)\sum_{j=1}^{k}\sum_{l=\lfloor \frac{\epsilon_{0}\epsilon_1 n}{2\log 2}\rfloor -1}^{\infty}e^{\eps \kappa_0 n} 2^{-l \eps_0/12}\\
	&\lesssim \sharp \mathcal{R}_n^{k+1}(\eps,\eps_0)\cdot k  e^{\epsilon\kappa_{0}n-\epsilon_{0}^2\epsilon_{1}n/24}.
\end{align*} 
Thus our desired bound holds.

\bigskip
\noindent \textbf{Step 2. Reducing the proof of \eqref{eq:triple} to establishing \eqref{eq:triple2}.}
Instead of proving \eqref{eq:triple} directly we will instead prove the following statement:

\medskip 
\textit{There exists $\kappa_0 > 0$ such that for all $\eps > 0$, $n \in \N$, $\eta \in J_n(\eps,\eps_0)$, $\sigma \in [(R|\eta|^{-1})^{1-\frac{\lfloor \epsilon_{0}n\rfloor}{n}},|\eta|^{-\eps_1}]$, $x \in I$ we have
	\begin{align}\label{eq:triple2}\sharp\{(\a,\b,\c) \in \cR_n(\eps,\eps_0)^3 : |e^{2\lambda n}|\varphi_{{\bf ab}}'(x)|-e^{2\lambda n}|\varphi_{{\bf ac}}'(x)|| \leq \sigma \} \lesssim e^{\kappa_0 \eps n} \sigma^{\eps_0/2} \sharp \cR_n(\eps,\eps_0)^3.\end{align}}
The difference between \eqref{eq:triple2} and \eqref{eq:triple} is that we consider $\sigma \in [(R|\eta|^{-1})^{1-\frac{\lfloor \epsilon_{0}n\rfloor}{n}},|\eta|^{-\eps_1}]$ instead of $\sigma \in  [R^{-2}|\eta|^{-1},|\eta|^{-\eps_1}].$ It can be shown that \eqref{eq:triple} follows from \eqref{eq:triple2}, albeit for a potentially different value of $\kappa_0$. We leave the details to the interested reader. We emphasise that in \eqref{eq:triple2} we observe a $\sigma^{\eps_0/2}$ term in our upper bound, whereas in \eqref{eq:triple} we observe a $\sigma^{\eps_0/3}$ term in our upper bound. This difference in the exponents is crucial when it comes to deriving \eqref{eq:triple} from \eqref{eq:triple2}. To complete the proof of our proposition it now suffices to prove \eqref{eq:triple2}.

\bigskip
\noindent \textbf{Step 3. Reducing the proof of \eqref{eq:triple2} to establishing \eqref{eq:scalem}.}
Let us fix $\epsilon,n,\eta$ and $\sigma$ as in its statement of \eqref{eq:triple2}. We fix $m \in \N$ such that 
$$e^{-\eps_0 (m-1)}< \sigma \leq e^{-\eps_0 m}.$$
Note that since $\sigma \in [(R|\eta|^{-1})^{1-\frac{\lfloor \epsilon_{0}n\rfloor}{n}},|\eta|^{-\eps_1}]$ and $\eta\in J_{n}(\epsilon,\epsilon_0)$, we have $$e^{-\epsilon_0(n-\lfloor \epsilon_0 n\rfloor)}\leq \sigma \leq e^{-\frac{\epsilon_0\epsilon_1n}{2}}$$ therefore
$$\frac{\eps_1}{2} n \leq m \leq n-\lfloor \epsilon_0 n\rfloor.$$
Appealing to a bounded distortions argument and the fact that $p_{\a}p_{\b}=\p_{\a\b}$ for any $\a,\b\in \A^{*}$, we can deduce that there exists a constant $\beta>0$ such that for every pair $(\a,\b) \in \cR_n(\eps,\eps_0)^2$, the concatenation $\a\b$ splits into a word $\a\b = \e \d$ with $\e := \a\b|_{2n-m} \in \cR_{2n-m}(\eps,\eps_0)$ and 
$$\d := \a\b|_{2n-m+1}^{2n} \in \tilde\cR := \Big\{\d \in \A^m :  [\d] \subset A_m(\beta \eps) \Big\}.$$
Thus if we now write
$$\cP = \{(\e,y) : y\pm e^{-\eps_0 m} \in [R^{-1},R]\quad \text{and}\quad \e \in \tilde \cR_{2n-m}(\eps) \}$$
we have:
\begin{align*}& \sharp\{(\a,\b,\c) \in \cR_n(\eps,\eps_0)^3 : |e^{2\lambda n}|\varphi_{{\bf ab}}'(x)|-e^{2\lambda n}|\varphi_{{\bf ac}}'(x)|| \leq \sigma \} \\
& \leq \sharp \cR_n(\eps,\eps_0) \sup_{\c \in \cR_n(\eps,\eps_0)} \sharp\{(\a,\b) \in \cR_n(\eps,\eps_0)^2 : |e^{2\lambda n}|\varphi_{{\bf ab}}'(x)|-e^{2\lambda n}|\varphi_{{\bf ac}}'(x)|| \leq \sigma \} \\
& \leq \sharp \cR_n(\eps,\eps_0) \sup_{\c \in \cR_n(\eps,\eps_0)} \sharp\{(\a,\b) \in \cR_n(\eps,\eps_0)^2 : |e^{2\lambda n}|\varphi_{{\bf ab}}'(x)|-e^{2\lambda n}|\varphi_{{\bf ac}}'(x)|| \leq e^{-\eps_0 m} \} \\
& \leq \sharp \cR_n(\eps,\eps_0) \sharp \cR_{2n-m}(\eps,\eps_0)  \sup_{(\e,y)  \in \cP}\sharp\{\d \in \tilde\cR : e^{2\lambda n}|\varphi_{{\bf e d}}'(x)| \in B(y,e^{-\eps_0 m})\}.
\end{align*}
Since $e^{-\eps_0 m} \sim \sigma$ and 
$$\sharp \tilde\cR \sim e^{\beta' \eps m} \sharp\cR_m(\eps,\eps_0)$$
for some $\beta' > 0$ depending on $\beta$ by \eqref{e:cards} and the definition of $A_{m}(\beta \eps)$, \eqref{eq:triple2} will follow if we can establish the following: 

\bigskip
\textit{There exists $\kappa_0 > 0$ such that for any $y \in \R$ with $y\pm e^{-\eps_0 m} \in [R^{-1},R]$, $\e \in \cR_{2n-m}(\eps,\eps_0)$ and $x \in I$ we have
\begin{align} \label{eq:scalem}\sharp\{\d \in \tilde\cR : e^{2\lambda n}|\varphi_{{\bf e d}}'(x)| \in B(y,e^{-\eps_0 m})\} \lesssim  e^{\eps \kappa_0 m} \sigma^{\eps_0/2} \sharp \tilde\cR.\end{align}}

\bigskip
\noindent \textbf{Step 4. Verifying \eqref{eq:scalem}.} Let us now proceed to prove \eqref{eq:scalem}.
As $y-e^{-\eps_0 m} \geq R^{-1} > 0$, we have 
$$e^{2\lambda n}|\varphi_{{\e \d}}'(x)| \in B(y,e^{-\eps_0 m})$$
 if and only if 
$$-\log |\varphi_{\e\d}'(x)| \in J := [2 \lambda n - \log(y+e^{-\eps_0 m}),2\lambda n - \log(y-e^{-\eps_0 m})].$$ 
so
$$ \sharp\{\d \in \tilde\cR: e^{2\lambda n}|\varphi_{{\e \d}}'(x)| \in B(y,e^{-\eps_0 m})\} =  \sharp\{\d \in \tilde\cR: -\log |\varphi_{\e\d}'(x)| \in J\}$$

We will bound this using a mollifier $h \in C^2(\R)$ satisfying
$$\chi_J \leq h, \quad \|h\|_1 \lesssim |J|, \quad \|h''\|_{L^1} \lesssim \frac{1}{|J|}.$$
As $\chi_J \leq h$, we have
$$ \sharp\{\d \in \tilde\cR :  -\log |\varphi_{\e\d}'(x)| \in J\}\leq \sum_{\d \in \tilde\cR }h(-\log |\varphi_{\e\d}'(x)|)^{1/2}$$
The Cauchy-Schwartz inequality implies then implies the bound
	\begin{align*}\sharp\{\d \in \tilde\cR :  -\log |\varphi_{\e\d}'(x)| \in J\}^2\leq  \Big(\sum_{\d \in \tilde\cR } p_\d h(-\log |\varphi_{\e\d}'(x)|)\Big)^2 \Big(\sum_{\d \in \tilde\cR}\frac{1}{p_\d }\Big)^2. \end{align*}
By Fourier inversion we know that
\begin{align*}
h(-\log |\varphi_{\e\d}'(x)|) = \int \exp(-2\pi i \xi \log |\varphi_{\e\d}'(x)|) \hat{h}(\xi)\,d\xi =  \int |\varphi_{\e\d}'(x)|^{-2\pi i \xi } \hat{h}(\xi)\,d\xi.
\end{align*}
Moreover, because the words in $\tilde \cR$ are all of length $m$, we have the bound:
\begin{align*} \sum_{\d \in \tilde \cR} p_\d  h(-\log |\varphi_{\e\d}'(x)|) &\leq \sum_{\d \in \A^m} p_\d  h(-\log |f_{\e\d}'(x)|) \\
&= \sum_{\d \in \A^m} p_\d  \int \hat{h}(\xi)|\varphi_{\e\d}'(x)|^{-2\pi i \xi } \,d\xi\\
&=\int \hat{h}(\xi)\sum_{\d \in \A^m} p_\d  |\varphi_{\e\d}'(x)|^{-2\pi i \xi } \,d\xi.
\end{align*}
Let $\Theta>1$ be such that Theorem \ref{thm:Linftybound} is satisfied for $r=0$ and $|b|>\Theta$. We now bound the integral above over the domains $|\xi| \leq \Theta/2\pi$ and  $|\xi| > \Theta/2\pi$.

Firstly, we have
\begin{align*} & \int_{|\xi| \leq \Theta / 2\pi} \hat{h}(\xi)\sum_{\d \in \A^m} p_\d  |\varphi_{\e\d}'(x)|^{-2\pi i \xi } \,d\xi \\
& \lesssim \sup_{|\xi|\leq \Theta/2\pi} \Big|\hat{h}(\xi)\sum_{\d \in \A^m} p_\d   |\varphi_{\e\d}'(x)|^{-2\pi i \xi } \Big|\\
& \leq \sup_{|\xi|\leq 1/2\pi} |\hat{h}(\xi)|\sum_{\d \in \A^m} p_\d  \\
& \leq |J| .
\end{align*}

Next, for the integral over $|\xi| > \Theta/2\pi$, for given such $\xi$, define the function
$$|g(x)| := |\varphi_\e'(x)|^{-2\pi i\xi}$$
and take $b := -2\pi \xi.$ Notice that the definition of the transfer operator gives that for any $x\in I$ we have the identity:
$$\sum_{\d \in \A^m} p_\d  |\varphi_{\e\d}'(x)|^{-2\pi i \xi } = \cL_{0+ib}^m (g)(x).$$
On the other hand, by Theorem \ref{thm:Linftybound}, there exists $\rho_0 > 0$ such that for all $m \in \N$ we have:
$$ \|\cL_{b}^m (g)\|_\infty \lesssim \rho_0^{m} |b|^{1/2} \|g\|_b.$$
By bounded distortions
$$|g(x)|= 1 \quad \text{and} \quad |g'(x)| = 2 \pi |\xi|  \frac{|\phi_\e''(x)|}{|\phi_\e'(x)|} |g(x)| \lesssim |\xi|,$$
so the $b$-norm is bounded:
$$ \|g\|_{b} = \|g\|_\infty + \frac{\|g'\|_{\infty}}{|b|} \lesssim 1.$$ 
Hence we can bound the integral over $|\xi| > \Theta/2\pi$ as follows:
\begin{align*}
&\int\limits_{|\xi|>\Theta/2\pi}  \hat{h}(\xi)\sum_{\d \in \A^m} p_\d  |\varphi_{\e\d}'(x)|^{-2\pi i \xi } \,d\xi \\
& \lesssim \int\limits_{|\xi|>\Theta/2\pi} |\hat{h}(\xi)| \cdot\rho_0^m |b|^{1/2}  \|g\|_{b} \,d\xi \\
& \lesssim \rho_0^m \int\limits_{|\xi|>\Theta/2\pi} |\hat{h}(\xi)| \cdot |\xi|^{1/2}\,d\xi.
\end{align*}
On the other hand, by integration by parts, we can bound $\hat{h}(\xi)$ for any $\xi \in \R$ as follows:
\[|\hat{h}(\xi)|\leq \frac{1}{1+|2\pi \xi|^2}(\|h\|_{L^1}+\|h''\|_{L^1}) .\]
Thus
\[\int_{|\xi|>\Theta / 2\pi} |\hat{h}(\xi)| \cdot |\xi|^{1/2}\,d\xi\leq \int \frac{|\xi|^{1/2}}{1+|2\pi \xi|^2}(\|h\|_{L^1}+\|h''\|_{L^1}) \,d\xi \lesssim \|h\|_{L^1}+\|h''\|_{L^1}. \]

Combining our bounds for the two integrals, we arrive at:
\begin{align*}
\sharp\{\d \in \tilde\cR: e^{2\lambda n}|\varphi_{{\e \d}}'(x) |\in B(y,e^{-\eps_0 m})\}^2 \lesssim  E(x) [\rho_0^m (\|h\|_{L^1}+\|h''\|_{L^1})+ |J|],
\end{align*}
where
$$E(x) := \sum_{\d \in \tilde\cR }\frac{1}{ p_\d }.$$
By definition of $\tilde \cR$, we have
$$E(x) \lesssim e^{\eps\beta m}  e^{h m} \sharp \tilde \cR.$$
So for some $\kappa > 0$
$$E(x) \lesssim e^{\eps \kappa n} \sharp \tilde \cR^2$$
Moreover, by the mean value theorem 
$$|J| \leq \frac{2e^{-\eps_0m}}{y-e^{-\eps_0 m}} \leq 2R e^{-\eps_0 m} \quad \text{and} \quad \frac{1}{|J|^3} \leq \frac{1}{8} R^{3} e^{3\eps_0 m}.$$
Recalling $R = e^{ \eps n}$ gives
$$ |J| \lesssim e^{ \eps n} e^{-\eps_0 m}$$
and
$$\frac{1}{|J|} \lesssim  e^{3 \eps n} e^{\eps_0 m}.$$
Then by the choice of $h$, we have
$$\|h\|_{L^1}+\|h''\|_{L^1} \leq |J| + \frac{1}{|J|} \lesssim e^{\eps n} e^{-\eps_0 m} + e^{3 \eps n} e^{\eps_0 m}.$$
Thus we obtain
\begin{align*}
\sharp\{\d \in \tilde\cR: e^{2\lambda n}|\phi_{{\e \d}}'(x)| \in B(y,e^{-\eps_0 m})\}^2 & \lesssim e^{\eps \kappa n} \sharp \tilde \cR^2 [ \rho_0^m  (e^{\eps n} e^{-\eps_0 m} + e^{3 \eps n} e^{\eps_0 m}) + e^{ \eps n} e^{-\eps_0 m}]\\
& \lesssim e^{\eps(3+ \kappa) n} \sharp \tilde \cR^2  \rho_0^m  e^{\eps_0 m}
\end{align*}
Since $\rho_0 \leq e^{-2\eps_0}$ by the choice of $\eps_0$ and $e^{-\eps_0 (m-1)}\leq \sigma,$ we see that 
$$\rho_0^m e^{\eps_0 m} \leq e^{-2\eps_0 m}e^{\eps_0 m} = e^{-\eps_0 m} \lesssim \sigma^{\eps_0}.$$
Selecting now $\kappa_0 = \frac{(3 + \kappa)}{\epsilon_1}$ and using that $n\leq \frac{2m}{\eps_1}$ gives us
$$\sharp\{\d \in \tilde\cR: e^{2\lambda n}|\phi_{{\e \d}}'(x)| \in B(y,e^{-\eps_0 m})\} \lesssim e^{\eps \kappa_0 m} \sigma^{\eps_0/2} \sharp \tilde \cR. $$
Thus the proof of \eqref{eq:scalem} is complete. 

This completes the proof of Proposition \ref{lma:distribution}.
\end{proof}

Combining Theorem \ref{lma:discretised} and Proposition \ref{lma:distribution}, we can now prove Theorem \ref{thm:main1}:

\begin{proof}[Proof of Theorem \ref{thm:main1}] Recall the data and notations we fixed in \ref{rmk:data}. Iterating the self-conformality 
$$\mu=\sum_{a\in \A}p_{a}\varphi_{a}\mu$$
yields using the notation $e(y) := \exp(-2\pi i y)$, $y \in \R$ that
$$\widehat{\mu}(\xi)= \sum_{\a\ast \b\in \A^{(2k+1)n}}p_{\a\ast \b}\int e(\xi\varphi_{\a\ast \b}(x))d\mu(x).$$ Here we have used the notation $\a\ast \b$ to mean $(a_0,b_1,a_1,b_2,\ldots,b_k,a_k)$ for $\a=(a_0,\ldots,a_k)\in \A^{(k+1)n}$ and $\b=(b_1,\ldots,b_k)\in \A^{kn}$. Now splitting this sum based upon whether a word is in $\mathcal{R}_{n}^{2k+1}(\epsilon)$ or not and using Lemma \ref{e:large deviation}, we have
$$|\widehat{\mu}(\xi)|\lesssim \left| \sum_{\a\in\mathcal{R}_n^{k+1}(\eps,\eps_0)}\sum_{\b\in\mathcal{R}_n^k(\eps,\eps_0)}p_{\a\ast \b}\int e(\xi\varphi_{\a\ast \b}(x))d\mu(x)\right|+ke^{-\delta n/2}.$$ Now taking squares and using the inequality $|a+b|^2\leq 2|a|^2+2|b|^2$ for $a,b\in \C$ yields 
\begin{equation}
	\label{e:Fourier splitting}
	|\widehat{\mu}(\xi)|^2\lesssim \left| \sum_{\a\in\mathcal{R}_n^{k+1}(\eps,\eps_0)}\sum_{\b\in\mathcal{R}_n^k(\eps,\eps_0)}p_{\a\ast \b}\int e(\xi\varphi_{\a\ast \b}(x))d\mu(x)\right|^2 +k^2e^{-\delta n}.
\end{equation} Focusing now on the first term, by the Cauchy-Schwartz inequality we have 
\begin{align*}
	&\left| \sum_{\a\in\mathcal{R}_n^{k+1}(\eps,\eps_0)}\sum_{\b\in\mathcal{R}_n^k(\eps,\eps_0)}p_{\a\ast \b}\int e(\xi\varphi_{\a\ast \b}(x))d\mu(x)\right|^2\\
	\leq& \left|\sum_{\a\in\mathcal{R}_n^{k+1}(\eps,\eps_0)}\sum_{\b\in\mathcal{R}_n^k(\eps,\eps_0)}p_{\a\ast \b}^2\right|\times \sum_{\a\in\mathcal{R}_n^{k+1}(\eps,\eps_0)}\sum_{\b\in\mathcal{R}_n^k(\eps,\eps_0)}\left|\int e(\xi\varphi_{\a\ast \b}(x))d\mu(x)\right|^2\\
\lesssim&  e^{(2k+1)\eps n}e^{-(2k+1)hn}\times \sum_{\a\in\mathcal{R}_n^{k+1}(\eps,\eps_0)}\sum_{\b\in\mathcal{R}_n^k(\eps,\eps_0)}\left|\int e(\xi\varphi_{\a\ast \b}(x))d\mu(x)\right|^2
\end{align*}
In the final line we've used that if $\a\in\mathcal{R}_n^{k+1}(\eps,\eps_0)$ and  $\b\in\mathcal{R}_n^k(\eps,\eps_0)$ then $p_{\a\ast\b}\leq e^{(2k+1) \eps n}\cdot e^{-(2k+1)hn},$ and $$\sum_{\a\in\mathcal{R}_n^{k+1}(\eps,\eps_0)}\sum_{\b\in\mathcal{R}_n^k(\eps,\eps_0)}p_{\a\ast \b}= \sum_{\a\in\mathcal{R}_n^{k+1}(\eps,\eps_0)}p_{\a}\sum_{\b\in\mathcal{R}_n^k(\eps,\eps_0)}p_{ \b}\leq  1.$$ Substituting the above into \eqref{e:Fourier splitting} proves
that for all $\epsilon>0$ we have
$$|\widehat{\mu}(\xi)|^2\lesssim e^{(2k+1)n \eps}e^{-(2k+1)hn}\cdot \sum_{\a\in\mathcal{R}_n^{k+1}(\eps,\eps_0)}\sum_{\b\in\mathcal{R}_n^k(\eps,\eps_0)}\left|\int e(\xi\varphi_{\a\ast \b}(x))d\mu(x)\right|^2 +k^2e^{-\delta n}.$$

Next, we will reduce the upper bound to exponential sums. We will use the fundamental theorem of calculus to replace the term $\varphi_{\a\ast \b}(x)-\varphi_{\a\ast \b}(y)$ with a product of derivatives. We begin by observing that for a fixed $\a\in \cR_{n}^{k+1}(\epsilon,\epsilon_0)$ we have
 \begin{align*}
	\sum_{\b\in\mathcal{R}_n^k(\eps,\eps_0)}\left|\int e(\xi\varphi_{\a\ast \b}(x))d\mu(x)\right|^2&=\sum_{\b\in\mathcal{R}_n^k(\eps,\eps_0)}\int\int e(\xi(\varphi_{\a\ast \b}(x)-\varphi_{\a\ast \b}(y)))\,d\mu(x)d\mu(y)\\
	&\leq \int\int\left| \sum_{\b\in\mathcal{R}_n^k(\eps,\eps_0)}e(\xi(\varphi_{\a\ast \b}(x)-\varphi_{\a\ast \b}(y)))\right|\,d\mu(x)d\mu(y)
\end{align*}
We now focusing on the final term in the above. For any $x,y$ we let $$\eta(x,y):= \xi e^{-2k\lambda n}(\varphi_{\a_k}(x)-\varphi_{\a_k}(y)).$$ Then by the regularity of $\a_k$, we have $$e^{-\eps n}e^{\epsilon_0 n}|x-y|\leq |\eta(x,y)|\leq e^{\eps n}e^{\epsilon_0 n}|x-y|.$$ 
Appealing to the fundamental theorem of calculus, and duplicating the arguments from \cite{SS} we have 
$$\left|\xi(\varphi_{\a\ast \b}(x)-\varphi_{\a\ast \b}(y))-\eta(x,y)\zeta_{1,\a}(b_1)\cdots \zeta_{k,\a}(b_k))\right|\lesssim e^{2k}e^{(k+2)\eps n}e^{-\lambda n}e^{\epsilon_0n}.$$
This in turn implies that
\begin{align*}
&\left|\sum_{\b\in\mathcal{R}_n^k(\eps,\eps_0)} e(\xi(\varphi_{\a\ast \b}(x)-\varphi_{\a\ast \b}(y)))\right|\\
\lesssim & \left|\sum_{\b\in\mathcal{R}_n^k(\eps,\eps_0)} e(\eta(x,y)\zeta_{1,\a}(b_1)\cdots \zeta_{k,\a}(b_k))\right|+\sum_{\b\in\mathcal{R}_n^k(\eps,\eps_0)}e^{2k}e^{(k+2)\eps n}e^{-\lambda n}e^{\epsilon_0n}\\
\lesssim& \left|\sum_{\b\in\mathcal{R}_n^k(\eps,\eps_0)} e(\eta(x,y)\zeta_{1,\a}(b_1)\cdots \zeta_{k,\a}(b_k))\right|+e^{2k}e^{(2k + 2) \eps n}e^{hk n}e^{-\lambda n}e^{\epsilon_0n}.
\end{align*}In the final line we have used that $\sharp \mathcal{R}_n^k(\eps,\eps_0)\leq e^{\eps n k }e^{khn}$.
Thus 
\begin{align*}
|\widehat{\mu}(\xi)|^{2}\lesssim &e^{(2k+1)\eps n}e^{-(2k+1)hn}\cdot \sum_{\a\in\mathcal{R}_n^{k+1}(\eps,\eps_0)}\iint \left|\sum_{\b\in\mathcal{R}_n^k(\eps,\eps_0)} e(\eta(x,y)\zeta_{1,\a}(b_1)\cdots \zeta_{k,\a}(b_k))\right|d\mu(x)d\mu(y)\\
+\,&e^{(2k+1)\eps n}e^{-(2k+1)hn}\sum_{\a\in\mathcal{R}_n^{k+1}(\eps,\eps_0)}e^{2k}e^{(2k+2)\eps n}e^{hk n}e^{-\lambda n}e^{\epsilon_0n}\\
+\,& k^2e^{-\delta n/2}.
\end{align*}
Now using the inequality $\sharp \mathcal{R}_n^{k+1}(\eps,\eps_0)\leq e^{(k+1)\eps n}e^{(k+1)hn},$ we see that the above implies
\begin{align*}|\widehat{\mu}(\xi)|^{2}\lesssim& e^{ (2k+1)\eps n}e^{-(2k+1)hn}\cdot \sum_{\a\in\mathcal{R}_n^{k+1}(\eps,\eps_0)}\iint\left|\sum_{b\in\mathcal{R}_n^k(\eps,\eps_0)} e(\eta(x,y)\zeta_{1,\a}(b_1)\cdots \zeta_{k,\a}(b_k))\right|\,d\mu(x)d\mu(y)\\
+&e^{2k}e^{ (5k+3)\eps n}e^{-\lambda n}e^{\epsilon_0n}+ k^2e^{-\delta n/2}.
\end{align*}
As $\eps_0 \leq \lambda/4$, and $\eps > 0$ is chosen small enough depending on $k$, the latter two terms go to zero exponentially in $n$, which in turn means polynomially in $|\xi|$. It remains to estimate the term with the integrals above.

By the self-conformality of $\mu$, and using the assumption that our IFS is non-trivial, we know that there exists $\kappa>0$ such that 
\begin{equation}
	\label{e:FengLaubound}
	\mu(B(x,r))\lesssim r^{\kappa}.
\end{equation}
 for any $x\in\mathbb{R}$ and $r>0$. The proof of this follows by adapting the proof of Feng and Lau \cite{FL} on overlapping self-similar measures. Using \eqref{e:FengLaubound} it follows that $$\mu\times \mu(\{(x,y):|x-y|\leq e^{\eps n}e^{-\epsilon_0n/2}\})\lesssim e^{\eps n \kappa}e^{-\epsilon_o\kappa n/2}.$$ Using this bound we have
\begin{align*}
&e^{\eps n (2k+1)}e^{-(2k+1)hn} \sum_{\a\in\mathcal{R}_n^{k+1}(\eps,\eps_0)}\iint \left|\sum_{\b\in\mathcal{R}_n^k(\eps,\eps_0)} e(\eta(x,y)\zeta_{1,\a}(b_1)\cdots \zeta_{k,\a}(b_k))\right|d\mu(x)d\mu_(y)\\
\leq &e^{\eps n (2k+1)}e^{-(2k+1)hn} \sum_{\a\in\mathcal{R}_n^{k+1}(\eps,\eps_0)}\iint\limits_{|x-y|\geq e^{\eps n}e^{-\epsilon_{0}n/2}} \left|\sum_{\b\in\mathcal{R}_n^k(\eps,\eps_0)} e(\eta(x,y)\zeta_{1,\a}(b_1)\cdots \zeta_{k,\a}(b_k))\right|d\mu(x)d\mu(y)\\
&+ e^{\eps n (2k+1)}e^{-(2k+1)hn} \sum_{\a\in\mathcal{R}_n^{k+1}(\eps,\eps_0)}\sum_{\b\in\mathcal{R}_n^k(\eps,\eps_0)}e^{\eps n \kappa}e^{-\epsilon_o\kappa n/2},
\end{align*}
which is bounded by
\begin{align*}
&e^{\eps n (2k+1)}e^{-(2k+1)hn} \sum_{\a\in\mathcal{R}_n^{k+1}(\eps,\eps_0)}\iint\limits_{|x-y|\geq e^{\eps n}e^{-\epsilon_{0}n/2}} \left|\sum_{\b\in\mathcal{R}_n^k(\eps,\eps_0)} e(\eta(x,y)\zeta_{1,\a}(b_1)\cdots \zeta_{k,\a}(b_k))\right|d\mu(x)d\mu(y)\\
&+ e^{\eps n (4k+2)}e^{\eps n \kappa}\epsilon^{-\epsilon_{0}\kappa n/2}.
\end{align*}
Here as $\eps > 0$ can be chosen small enough in terms of $k$ and $\kappa$, the second term here goes to zero exponentially in $n$, so polynomially in $|\xi|$. Thus we can just focus on bounding the integral term above.

If a pair of points $x$ and $y$ satisfies $|x-y|\geq e^{\eps n}e^{-\epsilon_{0}n/2}$, we have $\eta(x,y)\geq e^{\epsilon_0n/2}$ so $\eta(x,y) \in J_n(\eps,\eps_0)$. Hence
\begin{align*}
& e^{\eps n (2k+1)}e^{-(2k+1)hn} \sum_{\a\in\mathcal{R}_n^{k+1}(\eps,\eps_0)}\iint\limits_{|x-y|\geq e^{\eps n}e^{-\epsilon_{0}n/2}} \left|\sum_{\b\in\mathcal{R}_n^k(\eps,\eps_0)} e(\eta(x,y)\zeta_{1,\a}(b_1)\cdots \zeta_{k,\a}(b_k))\right|d\mu(x)d\mu(y)\\
&\qquad \leq  e^{\eps n (2k+1)}e^{-(2k+1)hn} \sum_{\a\in\mathcal{R}_n^{k+1}(\eps,\eps_0)}\sup_{\eta \in J_n(\eps,\eps_0)}\left|\sum_{\b\in\mathcal{R}_n^k(\eps,\eps_0)} e(\eta\zeta_{1,\a}(b_1)\cdots \zeta_{k,\a}(b_k))\right|.
\end{align*} 

Recall that in Proposition \ref{lma:distribution} we defined $\cW$ to be the set of $(k+1)$-tuples $\a \in \mathcal{R}_n^{k+1}(\eps,\eps_0)$ such that for all $j=1,\dots,k$, $\eta \in J_n(\eps,\eps_0)$ and $\sigma \in [R^{-2}|\eta|^{-1},|\eta|^{-\eps_1}]$, we have that
	$$ \sharp \{ ({\bf b, c}) \in \mathcal{R}_n(\eps,\eps_0)^2: |\zeta_{j,{\a}}({\b})-\zeta_{j,{\a}}({\c})| \leq \sigma \}  \leq \sharp \cR_n(\eps,\eps_0)^2 \sigma^{\eps_0/4}. $$
By Proposition \ref{lma:distribution} there exists $\kappa_0>0$ such that
$$ \frac{\sharp(\mathcal{R}_n^{k+1}(\eps,\eps_0)\setminus \mathcal{W})}{\sharp \mathcal{R}_n^{k+1}(\eps,\eps_0)} \lesssim  k  e^{\epsilon\kappa_{0}n-\epsilon_{0}^2\epsilon_{1}n/24}$$
Using this bound for $\cW$ together with $\sharp \mathcal{R}_n^{k+1}(\eps,\eps_0)\lesssim e^{\epsilon(k+1)n}e^{(k+1)hn}$ and $\sharp \mathcal{R}_n^{k}(\eps,\eps_0)\lesssim e^{\eps kn}e^{khn},$  we can deduce the following bound
\begin{align*}
&e^{\eps n (2k+1)}e^{-(2k+1)hn} \sum_{\a\in\mathcal{R}_n^{k+1}(\eps,\eps_0)}\sup_{\eta \in J_n(\eps,\eps_0)}\left|\sum_{\b\in\mathcal{R}_n^k(\eps,\eps_0)} e(\eta\zeta_{1,\a}(b_1)\cdots \zeta_{k,\a}(b_k))\right|\\
&\lesssim e^{(3k+2) \eps n}e^{-kh n}  \max_{\a\in\cW}\sup_{\eta \in J_n(\eps,\eps_0)}\left|\sum_{\b\in\mathcal{R}_n^k(\eps,\eps_0)} e(\eta\zeta_{1,\a}(b_1)\cdots \zeta_{k,\a}(b_k))\right| +ke^{(4k+2+\kappa_0)\epsilon n}e^{-\eps_0^2\epsilon_1n/24} .
\end{align*}
The second term in the above decays to zero exponentially in $n$, and therefore polynomially in $|\xi|,$ for $\epsilon$ sufficiently small. It remains to bound the first term.

Recall now that
$$R := R_{\eps,n} =  e^{\eps n}$$
and
$$\cZ_j := \cR_n(\eps,\eps_0), \quad \text{for all } j = 1,\dots,k$$
and the maps $\zeta_j := \zeta_{j,\a} : \cR_n(\eps,\eps_0) \to [R^{-1},R]$ will be defined by
$$\zeta_{j,\a}(\b) := e^{2\lambda n} |\varphi_{\a_{j-1} \b}'(x_{\a_j})|, \quad \b \in \cR_n(\eps,\eps_0),$$
If we now fix $\eta \in J_n(\eps,\eps_0)$, $\a \in \cW$ and $\sigma \in [R^{-2}|\eta|^{-1},|\eta|^{-\eps_1}]$, then as
$$ \sharp \{ (\b,\c)\in \mathcal{R}_n(\eps,\eps_0)^2 : |\zeta_{j,\a}(\b) - \zeta_{j,\a}(\c)| \leq \sigma\} \leq \sharp \cR_n(\eps,\eps_0)^2 \sigma^{\eps_0/4}$$
we have by Theorem \ref{lma:discretised} that
$$ e^{(3k+2) \eps n}e^{-kh n}\max_{\a\in \cW}\Big| \sum\limits_{{\b \in \mathcal{R}_n^k(\eps,\eps_0)}} e(\eta \zeta_{1,{\a}}({\b_1})...\zeta_{k,{\a}}({\b}_k)) \Big| \lesssim e^{(4k+2) \eps n}|\eta|^{-\eps_2} \lesssim e^{(4k+2) \eps n}e^{-\eps_0 \eps_2 n/2}$$
since $|\eta| \geq e^{\eps_0 n/2}$ by  the definition of $J_n(\eps,\eps_0)$ as $\eta \in J_n(\eps,\eps_0)$. By making sure that $\eps > 0$ is chosen small enough, we have proven that $\widehat{\mu}(\xi)$ decays to zero polynomially in $|\xi|$. 
\end{proof}

\section{Fractal Uncertainty Principles from Fourier decay}\label{sec:fup}

Finally, we give the details of the proof of how the Fractal Uncertainty Principle follows from Fourier decay. The proof method here is adapted from the proof of the Fractal Uncertainty Principle for the Patterson-Sullivan measure \cite{BD1,Dyatlov} adapted to the measures we have with weaker regularity. Fix $\mu_j$ as $(C_j^-,\delta_j^-,C_j^+,\delta_j^+,h)$-Frostman measures, $j = 1,2$, and denote their supports by $K_1$ and $K_2$.

 Given $h > 0$, define the semiclassical Fourier transform of $f : \R^d \to \C$ by:
$$\cF_h f(\xi) := \frac{1}{(2\pi h)^{d/2}} \int_{\R^d} e^{-i x \cdot \xi / h} f(x) \, dx, \quad \xi \in \R^d.$$
Then FUP follows for $X = K_1 + B(0,h)$ and $Y = K_2 + B(0,h)$ if we can prove:
$$\|\1_X \cF_h^* \1_Y\|_{L^2(\R^d) \to L^2(\R^d)} \lesssim h^{\beta}.$$
Define
$$\Upsilon_X^\pm (x) =  \frac{1}{4h^{\delta_1^\pm}} \mu_1(B(x,2h)) \quad \text{and} \quad \Upsilon_Y^\pm(y) =  \frac{1}{4h^{\delta_2^\pm}} \mu_2(B(y,2h))$$
Then for all $x \in X$ we have
$$C_1^- \leq \Upsilon_X^-(x),  \Upsilon_X^+(x) \leq C_1^+,$$
and similarly for all $y \in Y$ we have
$$C_2^- \leq \Upsilon_Y^-(x),  \Upsilon_Y^+(x) \leq C_2^+$$
Indeed, for example if $x \in X$, then for some $x_0 \in K_1$ we have $|x-x_0| \leq h$. Thus $\mu_1(B(x,2h)) \geq \mu_1(B(x_0,h)) \geq C_1^- h^{\delta_1^-}$.

\begin{lemma}
Suppose for any bounded $u : \R^d\to \C$ we have
$$\|\sqrt{\Upsilon_X^-} \cF_h^* \Upsilon_Y^- u \|_{L^2(\R^d)} \lesssim h^{\beta} \|\sqrt{\Upsilon_Y^-} u\|_{L^2(\R^d)}.$$
Then
$$\|\1_X \cF_h^* \1_Y\|_{L^2(\R^d) \to L^2(\R^d)} \lesssim h^{\beta}.$$
\end{lemma}

\begin{proof}
Let $f \in L^2(\R^d)$ be bounded. Write $u = \frac{f\1_Y}{\Upsilon_Y^-}$. Then $u$ is bounded as $\Upsilon_Y^-$ is bounded from below. Then by definition
$$\|\1_X \cF_h^* \1_Y f\|_{L^2(\R^d)} = \|\1_X \cF_h^* \Upsilon_Y^- u\|_{L^2(\R^d)}.$$
Moreover, as $C_1^- \lesssim \Upsilon_X^-$ we obtain
$$\|\1_X \cF_h(\Upsilon_Y^- u)\|_{L^2(\R^d)} \lesssim \|\Upsilon_X^- \cF_h \Upsilon_Y^- u\|_{L^2(\R^d)} \lesssim h^{\beta} \|\sqrt{\Upsilon_X^-} u\|_{L^2(\R^d)}$$
by the assumption. Finally, 
$$\|\sqrt{\Upsilon_Y^-} u\|_{L^2(\R^d)} = \|  f \1_Y \|_{L^2(\R^d)} \leq \|  f  \|_{L^2(\R^d)}.$$
Bounded $L^2(\R^d)$ functions generate $L^2(\R^d)$, so this implies the claim for all $L^2(\R^d)$ functions.
\end{proof}

We will now prove
\begin{prop}\label{prop:mainFUP}
For any bounded $u : \R^d\to \C$ we have
$$\|\sqrt{\Upsilon_X^-} \cF_h^*\Upsilon_Y^- u \|_{L^2(\R^d)} \lesssim h^{\beta} \|\sqrt{\Upsilon_Y^-} u\|_{L^2(\R^d)}.$$
\end{prop}

\begin{proof}
Fix a bounded $u : \R^d \to \C$. For $x \in \R^d$ and $t \in \R^d$ define the translation of any $v : \R^d \to \R$ by:
$$\omega_t v (x) = v(x - t).$$
Then by Fubini's theorem
$$\|\sqrt{\Upsilon_X^-} \cF_h^*\Upsilon_Y^- u \|_{L^2(\R^d)}^2  = \frac{1}{4h^{\delta_1^-}} \int_{B(0,2h)} \|w_t \cF_h^{\ast} \Upsilon_Y^- u\|_{L^2(\mu_1)}^2 \, dt$$
and
$$\|\sqrt{\Upsilon_Y^-} u\|_{L^2(\R^d)}^2 = \frac{1}{4h^{\delta_2^-}}\int_{B(0,2h)} \|\omega_s v\|_{L^2(\mu_{2})}^2 \, ds.$$
Also
$$\omega_t \cF_h^* \Upsilon_Y^- u(x) = \frac{1}{4 h^{d/2 + {\delta_2^-}}} \int_{B(0,2h)} \int e^{2\pi i (x-t) \cdot (y-s) / h} \omega_s u(y) \, d\mu_{2}(y) \, ds.$$
Define an operator:
$$B_t u (x) = \int e^{2\pi i (x-t)\cdot y / h} u(y) \, d\mu_{2}(y)$$
so
$$ \int e^{2\pi i (x-t)\cdot (y-s) / h} \omega_s u(y) \, d\mu_{2}(y)  = e^{-2\pi i(x-t) \cdot s / h} B_t (\omega_s u)(x).$$
Thus
\begin{align*}\|\sqrt{\Upsilon_X^-} \cF_h^*\Upsilon_Y^- u \|_{L^2(\R^d)}^2 &= \frac{1}{4h^{\delta_1^-}} \int_{B(0,2h)} \|w_t \cF_h^{\ast} \Upsilon_Y^- u\|_{L^2(\mu_{1})}^2 \, dt \\
&\lesssim h^{d-{\delta_1^-}-2{\delta_2^-}} \sup_{|t| \leq 2h} \int_{B(0,2h)} \| B_t(\omega_s u)\|_{L^2(\mu_{1})}^2 \, ds.
\end{align*}
Now we are done if we can show 
$$\|B_t v\|_{L^2(\mu_{K_1})} \lesssim h^{\alpha/2} \|v\|_{L^2(\mu_{K_1})}$$
for any bounded $v$ and any $|t| \leq 2h$. Indeed, then 
\begin{align*}h^{d-{\delta_1^-}-2{\delta_2^-}}\sup_{|t| \leq 2h} \int_{B(0,2h)} \| B_t(\omega_s u)\|_{L^2(\mu_{1})}^2 \, ds &\lesssim h^{d-{\delta_1^-}-2{\delta_2^-}} \sup_{|t| \leq 2h} \int_{B(0,2h)} \| \omega_s u\|_{L^2(\mu_{1})}^2 \, ds\\
& \lesssim h^{d-{\delta_1^-}-{\delta_2^-} + \alpha/2}\|\sqrt{\Upsilon_Y^-} u\|_{L^2(\R^d)}^2.
\end{align*}
We can do this step now. By the $TT^*$ theorem \cite{GV}, we have
$$\|B_t\|_{L^2(\mu_{1}) \to L^2(\mu_{1})}^2 = \|B_tB_t^*\|_{L^2(\mu_{1}) \to L^2(\mu_{1})}$$
and we can write
$$B_tB_t^* v(z) = \int K(z,z') v(z') \, d\mu_{2}(z')$$
with kernel
$$K(z,z') := \int e^{-2\pi i (z-z')y/h} \, d\mu_{2}(y).$$
By Schur's inequality for $K(z,z')$, we obtain:
$$\|B_t\|_{L^2(\mu_{1}) \to L^2(\mu_{1})}^2 \leq \sup_{z \in K_1} \int |K(z,z')| \, d\mu_2(z').$$
By Fourier decay assumption on $\mu_2$, we have:
$$|K(z,z')| = \Big|\widehat{\mu_2}\Big(\frac{z-z'}{h}\Big)\Big| \lesssim \Big|\frac{z-z'}{h}\Big|^{-\alpha}.$$
Splitting
$$ \int |K(z,z')| \, d\mu_2(z) = \int_{|z-z'| \geq h^{1/2}} |K(z,z')| \, d\mu_2(z) + \int_{|z-z'| \leq h^{1/2}}|K(z,z')| \, d\mu_2(z) ,$$
we get that the first integral is bounded by $\lesssim h^{\alpha/2}$ and the second integral by $\lesssim h^{\delta_2^+/2}$ using the Frostman assumption of $\mu_2$. Finally, we know that $2\alpha \leq \delta_2^+$, so the claim is done.
\end{proof}

\section*{Acknowledgements}

We thank Amir Algom, Semyon Dyatlov, Jonathan Fraser, Antti K\"aenm\"aki, Connor Stevens, Sascha Troscheit and Meng Wu for useful discussions during the preparation of this manuscript, in particular Amir Algom for coordinating the submission of this and the independent work \cite{AHW3} simultaneously.

\bibliographystyle{plain}

\vspace{0.1in}

\end{document}